\documentclass[11pt,reqno]{amsart}
\usepackage{amsmath,amssymb,amscd,amsthm,amsxtra,amsfonts}
\usepackage{mathtools,mathrsfs,dsfont,xparse}
\usepackage{graphicx,epstopdf}
\usepackage{multirow}
\usepackage{cases}
\usepackage{enumerate}
\usepackage[]{hyperref}

\usepackage{color}

\allowdisplaybreaks[4]

\theoremstyle{plain}
\newtheorem{thm}{Theorem}[section]
\newtheorem{prop}[thm]{Proposition}
\newtheorem{lem}[thm]{Lemma}
\newtheorem{cor}[thm]{Corollary}

\theoremstyle{definition}
\newtheorem{defn}[thm]{Definition}
\newtheorem{rmk}[thm]{Remark}
\newtheorem{fact}[thm]{Fact}
\newtheorem{claim}[thm]{Claim}

\numberwithin{equation}{section}

\newcommand{\D}{\mathbf{D}}
\newcommand{\0}{\mathbf{0}}
\newcommand{\h}{\mathbb{H}}
\newcommand{\M}{\mathcal{M}}
\newcommand{\OO}{\mathcal{O}}
\newcommand{\NN}{\mathcal{N}}

\newcommand{\R}{\mathbb{R}}

\newcommand{\Z}{\mathbb{Z}}
\newcommand{\T}{\mathbb{T}}

\DeclareMathOperator{\re}{Re}
\DeclareMathOperator{\im}{Im}

\DeclareDocumentCommand{\abs}{s m}{
  \operatorname{}
  \IfBooleanTF{#1}{#2}{\left|#2\right|}}

\DeclareDocumentCommand{\norm}{s m}{
  \operatorname{}
  \IfBooleanTF{#1}{#2} {\left\| #2\right\|}}

\DeclareDocumentCommand{\inner}{s m}{
  \operatorname{}
  \IfBooleanTF{#1}{#2}{\left \langle#2\right \rangle}}

\DeclareDocumentCommand{\parenthese}{s m}{
  \operatorname{}
  \IfBooleanTF{#1}{#2}{\left(#2\right)}}

\DeclareDocumentCommand{\square}{s m}{
  \operatorname{}
  \IfBooleanTF{#1} {#2}{\left[#2\right]}}

\DeclareDocumentCommand{\bracket}{s m}{
  \operatorname{}
  \IfBooleanTF{#1}{#2}{\left\{#2\right\}}}

\begin{document}

\author[Staffilani and Yu]{Gigliola Staffilani and Xueying Yu}

\address{Gigliola Staffilani
\newline \indent Department of Mathematics, MIT\indent 
\newline \indent  77 Massachusetts Ave, Cambridge, MA 02139,\indent }
\email{gigliola@mit.edu}
\thanks{$^3$  G.S. is  funded in part by  DMS-1764403}

\address{Xueying Yu
\newline \indent Department of Mathematics, MIT\indent 
\newline \indent  77 Massachusetts Ave, Cambridge, MA 02139,\indent }
\email{xueyingy@mit.edu}

\title[On the high-low method for NLS on the hyperbolic space]{On the high-low method for NLS on the hyperbolic space}
\begin{center}

\maketitle

\small{\it Dedicated to the memory of Jean Bourgain}

\begin{abstract}
In this paper, we first prove that the cubic, defocusing nonlinear Schr\"odinger equation on the two dimensional hyperbolic space with radial initial data in $H^s(\h^2)$ is globally well-posed and scatters when $s > \frac{3}{4}$. Then we extend the result to nonlineraities of order $p>3$.   The result is proved by extending the   high-low method of Bourgain in the hyperbolic setting and by using a Morawetz type estimate proved by the first author and Ionescu.
\end{abstract}

\end{center}

\section{Introduction}

In this paper we consider the cubic nonlinear Schr\"odinger (NLS) initial value problem on the hyperbolic plane $\h^2$:
\begin{align}\label{NLS}
\begin{cases}
i \partial_t u + \Delta_{\h^2} u = \abs{u}^2 u, & t \in \R , \quad  x \in \h^2,\\
u(0,x) = \phi(x), & 
\end{cases}
\end{align}
where $u=u(t,x)$ is a complex-value function in spacetime $\R \times \h^2$ and $\phi$ is a radial initial datum.

The solution of \eqref{NLS} conserves both the mass:
\begin{align}\label{eq Intro Mass}
M(u(t)) : = \int_{\h^2} \abs{u(t,x)}^2 \, dx = M(u_0) ,
\end{align}
and the energy:
\begin{align}\label{eq Intro Energy}
E(u(t)) :  = \int_{\h^2} \frac{1}{2} \abs{\nabla_{\h^2} u(t,x)}^2 + \frac{1}{4} \abs{u(t,x)}^{4} \, dx = E(u_0) .
\end{align}
Conservation laws of mass and energy give the control of the $L^2$ and $\dot{H}^1$ norms of the solutions, respectively.

Our goal in this paper is to prove the global well-posedness and scattering of \eqref{NLS} with the regularity of the initial data below $H^1$.

In order to best frame the problem and  to emphasize its  challenges  we start by recalling  the results in $\R^d$, a setting  that has been extensively considered in recent years.
Consider the  evolution equation in \eqref{pNLS} with general non-linearities
\begin{align}\label{pNLS}
i \partial_t u + \Delta u = \abs{u}^{p-1} u, \quad p >1
\end{align}
in  $\R^d$. Let us first recall that the critical scaling exponent  in $\R^d$ is
\begin{align}\label{scrit}
s_c:=\frac{d}{2} - \frac{2}{p-1}.
\end{align}
It is well-known that in the sub-critical and critical regimes ($ s> s_c$ and $ s= s_c$, respectively), the initial value problem \eqref{pNLS} is locally well-posed{\footnote{With {\it local well-posedness} we mean local in time existence, uniqueness and continuous dependence of the data to solution map.}},  \cite{CW1, CW2, CW3, Ca}. Thanks to the conservation laws of energy and mass, the $H^1$-subcritical initial value problem and the $L^2$-subcritical initial value problem are globally well-posed in the energy space $H^1$ and mass space $L^2$, respectively. The questions about scattering{\footnote{This will be made more precise later, see for example Theorem \ref{thm Main H^2}, but in general terms, with {\it scattering} we intend that the nonlinear solution as time goes to infinity approachs a  linear one.} are much more delicate.

Before we talk about global results in the more general subcritical case with data  with regularity between $L^2$ (mass) and $H^1$ (energy), that is in $H^s, \, 0<s<1$, let us denote  with $g_{\M}^p$ the regularity  index above which one obtains global well-posedness  for the NLS problem on the manifold $\M$ with power nonlinearity  $p$, and  by $s_{\M}^p$ the index above which one obtains scattering (with the global well-posedness) again  on the manifold $\M$ with  power nonlinearity $p$.

The very first global well-posedness result in the subcritical case between the two  (mass and energy) conservation laws ($0<s<1$) was given  by Bourgain in \cite{B1}, where he developed the high-low method to prove  global well-posedness for  the cubic $(p=3) $ NLS in two dimensions for initial data in $H^s, \, s > \frac{3}{5}$. According to the above notation  the regularity index in \cite{B1} is $g_{\R^2}^3 = \frac{3}{5}$.

We now describe the high-low method of Bourgain because it is the inspiration for part of our current work. To start the initial datum is decomposed into a (smoother) low frequency part and a (rougher) high frequency part. The first step is to solve the NLS globally for the smoother part, for which the energy is finite,  and then solve a difference equation for the rougher part. The {\it miracle} in this argument, that allows one to continue with an iteration, is that in fact the Duhamel term in the solution to the difference equation is small and smoother in an interval of time that is inverse proportional to the size of the low frequency  part of the initial datum. At the next iteration one merges this smoother part with the evolution of the low frequency part of the datum and repeats. It is worth mentioning that in order to obtain the {\it miracle} step, Bourgain used a Fourier transform based space $X^{s,b}$  \cite{B1} that captures particularly well the behavior of solutions with low regularity initial datum. Let us remark that there is no scattering result in the high-low method proposed by Bourgain.

In \cite{CKSTT1}, Colliander-Keel-Staffilani-Takaoka-Tao improved the global well-posedness index $g_{\R^2}^3$ of the initial data to $\frac{4}{7}$ by introducing a different method, now known as I-method. This is also based on an  iterative argument. One first defines a Fourier multiplier that smooths out the initial data into the energy space and proves that the energy of the smoothed solution is almost conserved, that is, at each iteration the growth of such modified energy is uniformly small. The index $g_{\R^2}^3$ is derived by keeping the accumulation of energy controlled. As a result, in \cite{CKSTT1} the authors obtained a polynomial growth of the sub-energy Sobolev norm of the global solution. The cubic NLS in $\R^3$ was also considered in \cite{CKSTT1} and the index  $g_{\R^3}^3 = \frac{5}{6}$.

Later, in \cite{CKSTT2} by combining the Morawetz estimate with the I-method and a bootstrapping argument, the same authors were able to lower the index $g_{\R^3}^3$ to $\frac{4}{5}$ and proved, for the first time\footnote{Actually in \cite{B2} Bourgain proved the global well-posedness  for general data with index $g_{\R^3}^{3} = \frac{11}{13}$ and  scattering for radial data with index $s_{\R^3}^{3} = \frac{5}{7}$.}, that  the global solution also scatters, hence  $g_{\R^3}^3=s_{\R^3}^3=\frac{4}{5}$. To prove scattering, one needs to show that a  spacetime norm of the solution is uniformly bounded. To this end, an iteration of local well-posedness would not suffice. Instead, one uses a Morawetz estimate that gives a uniform bound of the  $L^4$ spacetime norm of the solution, combined with   the I-method. More in details one splits the time line into a finite number of intervals $I_j$, of possible infinite length,  on which the $L^4_{I_j}$ of the solution is  small. The smallness  allows for a  better spacetime bound of the global solution on each interval $I_j$, and then one uses  an iteration on the finite number of these intervals, which finally gives the desired spacetime uniform bound for the solution and hence    scattering.  

More results on the high-low method and the I-method both in $\R^d$ or compact manifolds can be found in \cite{CKSTT2, CR, DPST1, DPST2, DPST3, D1, D3, D4, D7,  FG, GC, Ha, LWX, Su, Tz}.

We now consider the initial value problem  
\begin{align}\label{NLSHp}
\begin{cases}
i \partial_t u + \Delta_{\h^d} u = \abs{u}^{p-1} u, & t \in \R , \quad x \in \h^d,\\
u(0,x) = \phi(x), & 
\end{cases}
\end{align}
with $p>1$.  Compared to what we recalled above, we expect even   better results in $\h^d$. In fact the  negative curvature of the ambient manifold allows for  more dispersion in $\h^d$ than in the Euclidean spaces. Mathematically we can see this in  the Strichartz estimates on $\h^d$, a family of estimates that is  broader than the one obtained for  the Euclidean space, see  \cite{AP, IS}. The fact that the family of Strichartz estimates is larger in $\h^d$ reminds us of another case in which this is true. In fact also for  the wave equation  the Strichartz estimates form  a larger family. In this case though it is not the curvature of the ambient manifold that generates a larger number of estimates,  but instead it is the fact that the wave operator has a strong smoothing effect pointwise in time, a property  that is not enjoyed by  the Schr\"odinger operator. As a consequence when one considers a nonlinear wave equation, the smoother and more plentiful  estimates provide more suitable control of the nonlinear terms, and this is the reason why in the nonlinear wave setting, the {\it miracle} step in the high-low method in \cite{KPV} does not need the Fourier type spaces $X^{s,b}$ mentioned above. However, in contrast, the larger range of the Strichartz estimates for the Schr\"odinger operator  in the hyperbolic space still is not readily  enough  to handle the  {\it miracle} step since  although one obtains better spacetime estimates, there is no pointwise  smoothing effect, hence the context we work in  is more challenging than the one in \cite{KPV}. At this point  one may guess that using some hyperbolic version of the space $X^{s,b}$ may do the trick.  While this is indeed the case when the problem is posed in $\T^d$, see for example \cite{DPST2},  in $\h^d$ the space it is not clear how to define the Fourier transform based $X^{s,b}$ type spaces in a way that is useful to handle nonlinearities. A naive definition using the Helgason-Fourier transform in \cite{H} is deficient because of the following two reasons:  first, the eigenfunctions of the Laplace-Beltrami operator on $\h^d$ lead to a very different Fourier inversion formula and Plancherel theorem. In particular we cannot claim that the Fourier transform of a product is   the  convolution of Fourier transforms, which is a fundamental fact used in the estimates of nonlinear terms via the space $X^{s,b}$. Second, the frequency localization based on the Helgason Fourier transform does not behave well in $L^p(\h^d)$, which causes  difficulties  in defining  an effective Littlewood-Paley decomposition. We anticipate here that our approach to  recover the {\it miracle} step, where one has to prove a gain of  smoothness for the solution to the nonlinear difference equation,  takes advantage of a Kato type smoothing effect. This smoothing is  not pointwise in time, like for the wave operator, but in average in time, hence much weaker. In oder to make up for this weakness we need to use a maximal function estimate combined with a better Sobolev embedding, which in turn forces us to assume  radial symmetry for our initial data. We expect though that our global well-posedness and scattering results are true in general and we believe that the more sophisticated smoothing effect in \cite{LLOS} may play an important role.

We now move to a summary of results that have been proved in the context of well-posedness and scattering for NLS in $\h^d$.  Although the initial value problem \eqref{NLSHp} cannot be properly scaled, we still use the same index $s_c$ defined in \eqref{scrit} to guide us in gauging the difficulty of proving global well-posedness and scattering for \eqref{NLSHp}. The subcritical initial value problem in the hyperbolic setting was first considered in \cite{BCS}, where the authors  proved  scattering for a family of power-type nonlinearity NLS with radial $H^1$ data.  In \cite{BD} the authors showed global well-posedness, scattering and blow-up results for energy-subcritical focusing{\footnote{An NLS is called focusing when the nonlinearity in \eqref{pNLS} has a negative sigh, that is, $i \partial_t u + \Delta_{\h^2} u =- \abs{u}^{p-1} u$.}} NLS also on the hyperbolic space. 
In the critical setting, in \cite{IPS} the authors proved  global well-posedness and scattering of the energy-critical NLS in $\h^3$.  This result uses an ad hoc profile decomposition technique to transfer the already available result in $\R^3$ \cite{CKSTT4} into the $\h^3$ setting. A similar technique was used also in $\T^3$ for the same  energy critical problem  \cite{IP}. We do not think that this method, which is well suited for critical settings,  could work in our subcritical setting, when  the initial data are in $H^s(\h^2), \, 0<s<1$, but it may  work to transfer in $\h^2$ the result that  Dodson proved for mass critical in $\R^2$ \cite{D6}. 
To the best of the authors' knowledge, there are no known subcritical global well-posedness and scattering results with initial data not at the conservation law level in hyperbolic spaces.

We now state the  main result of this work for the initial value problem  \eqref{NLS}. Later in Section \ref{general} we state a similar result for the more general version  \eqref{NLSHp} with $p>3$.
\begin{thm}\label{thm Main H^2}
The initial value problem \eqref{NLS} with radial initial data $\phi \in H^s(\h^2)$ with $s > \frac{3}{4} $ is globally-well-posed and scattering holds, that is   there exists $u_{\pm} \in H^s (\h^2) $  such that 
\begin{align}
\lim_{t \to \pm \infty}  \norm{u(t) -e^{it \Delta_{\h^2}} u_{\pm}}_{H_x^s (\h^2)} =0.
\end{align}
\end{thm}

\begin{rmk} Here we conduct a discussion on the indices of regularity  for global well-posedness and we make a comparison with other results.

As discussed above,  the equivalent case we consider here but in $\R^2$ was treated by Bourgain without redial symmetry using the $X^{s,b}$ space. Since we cannot use the same approach in $\h^2$, we   decided first to rework this case using different tools such as Kato smoothing effect, maximal function estimates and better Sobolev embedding. We did this because in $\R^2$ we have  a Littlewood-Paley decomposition that works very well. Using these tools  in the implementation of the high-low method, we obtained that the cubic radial NLS is globally well posed when $s > \frac{4}{5}$, that is $g_{\R^2}^3 = \frac{4}{5}$ (see Theorem \ref{thm Main R^2} in the Appendix). Recall that Bourgain's result gives $g_{\R^2}^3=\frac{3}{5}$, which is better than what we can do in $\R^2$, and it is for general data. But what we achieved  in this first step is a blue print that is generalizable to the $\h^2$ space.

One  notes that  the  index $s_{\h^2}^3 = \frac{3}{4}$  that we attained in Theorem \ref{thm Main H^2} is smaller than the one we obtained in $\R^2$, where we worked out only the global well-posedness, not the scattering. This is because of the better radial Sobolev embedding in $\h^2$ and of the help coming from  the strong Morawetz estimate used in the local theory.

Now a little bit of history concerning  the indices of regularity for global well-posedness. 
In Bourgain's paper, where the high-low method was introduced \cite{B1}, the global existence index is $g_{\R^2}^3 = \frac{3}{5}$. Later in \cite{CKSTT1}, where  the  I-method was used, the index $g_{\R^2}^{3}$ was improved to $ \frac{4}{7}$, and later in  \cite{CKSTT3},  thanks to a sophisticated treatment of the Fourier multiplier involved in the  I-mehtod, the global existence index $g_{\R^2}^{3} $ was lowered further to $\frac{1}{2}$. Global well-posedness  of cubic NLS in two dimensions with $H^{\frac{1}{2}}$ data was proved in \cite{FG}. Also $g_{\R^2}^{3}$ was improved to $\frac{1}{3}$ in \cite{CR}  and to $ \frac{1}{4}$ in \cite{D1}. In \cite{D2, D5, D6}, Dodson proved  global well-posedness and scattering for the mass-critical NLS in any dimension. Also as a consequence, via  the persistence of regularity property, mass-critical NLS equations with any subcritical initial data are globally well-posed as well. 
\end{rmk}

\subsection{Blue print of the proof} In this subsection we summarize the main three of  the proof of the main Theorem \ref{thm Main H^2}.  In general terms we  combine the high-low method with a  Morawetz type estimate that gives a bound for the spacetime $L^4$ norm. 

The first part of the proof  deals with the analysis of the {\it energy increment}. 
Following Bourgain's high-low method, we first decompose the initial datum into a high and a low frequency part. Then we write the solution $u$ as the sum of the linear evolution of the high frequency part and a reminder $\zeta$ that solve a difference equation that evolves from the low frequency part of the original initial datum. In this first step we assume that in an interval $[0,\tau]$,  where $\tau$ could be infinity, the $L^4$ spacetime norm of the solution is small. We then prove an estimate for the energy increment of $\zeta$. This is the content of Proposition \ref{prop Local}. To prove this energy increment estimate we further decompose $\zeta=\zeta_1+\zeta_2$, where $\zeta_1$ is the nonlinear solution starting from the low frequency part of the datum  and $\zeta_2$ solves a difference equation with zero datum. This part is similar to the high-low method of Bourgain, but here the interval of time is not small, the smallness comes from the $L^4$ norm. The {\it miracle} step is then to be able to show that $\zeta_2$  is smoother and small in the appropriate norms. 
In the second part of the proof we assume that the total $L^4$ spacetime norm of the solution is bounded and we subdivide the time line into finitely many  intervals in which this norm is small. Here we apply the first part describe above and we prove a global energy increment for $\zeta$, this is Proposition \ref{prop Global}. In the last part we use a bootstrapping argument to show that indeed the $L^4$ norms of the solution $u$ is bounded. This part requires a modification of the Morawetz estimate  in \cite{IS}, see  Proposition \ref{prop Morawetz}, and it uses the global energy increment proved in Proposition \ref{prop Global}.

To summarize, the rest of this paper is organized as follows.
In Section \ref{sec Preliminaries}, we discuss the geometry of the domain $\h^2$ and collect the useful analysis tools in $\h^2$. In Section \ref{sec Energy increment}, we present the calculation of the energy increment of the smoother part of the solution. Next, in Section \ref{sec Morawetz}, we prove a modified Morawetz estimate, which will be used in Section \ref{sec GWP+S}. Finally, in Section \ref{sec GWP+S}, we run a bootstrapping argument based on the estimates derived from Sections \ref{sec Energy increment} and \ref{sec Morawetz} and complete the proof of Theorem \ref{thm Main H^2}.

\section*{Acknowledgement} G.S and X.Y. graciously acknowledge the support received by the Jarve Seed Fund. G.S. was also supported in part by  the grant NSF DMS-1764403,  and X.Y by an AMS-Simons travel grant. Both authors would like to thank A. Lawrie and S. Shahshahani for very insightful conversations.

\section{Preliminaries}\label{sec Preliminaries}
\subsection{Notations}
We define
\begin{align*}
\norm{f}_{L_t^q L_x^r (I \times \h^2)} : = \square{\int_I \parenthese{\int_{\h^2} \abs{f(t,x)}^r \, dx}^{\frac{q}{r}} dt}^{\frac{1}{q}},
\end{align*}
where $I$ is a time interval.

We use the Japanese bracket notation in the following sense:
\begin{align*}
\norm{\inner{\Omega} f}_{X} = \norm{ f}_{X} + \norm{\Omega f}_{X},
\end{align*}
where $X$ is one of the  normed spaces we use below. 

We adopt the usual notation that $A \lesssim  B$ or $B \gtrsim A$ to denote an estimate of the form $A \leq C B$ , for some constant $0 < C < \infty$ depending only on the a priori fixed constants of the problem.

\subsection{Geometry of the domain $\h^2$}
We consider the Minkowshi space $\R^{2+1}$ with the standard Minkowski metric
\begin{align*}
- (dx^0)^2 + (dx^1)^2 + (dx^2)^2
\end{align*}
and we define the bilinear form on $\R^{2+1} \times \R^{2+1}$,
\begin{align*}
\square{x,y} = x^0 y^0 -x^1 y^1 -x^2 y^2.
\end{align*}
The hyperbolic space $\h^2$ is defined as
\begin{align*}
\h^2 = \{ x \in \R^{2+1} : \square{x,x} =1 \text{ and } x^0 > 0 \}.
\end{align*}

An alternative definition for the hyperbolic space is 
\begin{align*}
\h^2 = \{ x = (t,s) \in \R^{2+1} , (t,s) = (\cosh r, \sinh r\omega), r \geq 0, \omega \in \mathbb{S}^1 \}.
\end{align*}
One has
\begin{align*}
dt = \sinh r \, dr, \quad ds = \cosh r\omega \, dr + \sinh r \, d \omega
\end{align*}
and the metric induced on $\h^2$ is
\begin{align*}
dr^2 + \sinh^2 r \, d\omega^2 ,
\end{align*}
where $d\omega^2$ is the metric on the sphere $\mathbb{S}^1$.

Then one can rewrite integrals as 
\begin{align*}
\int_{\h^2} f(x) \, dx= \int_0^{\infty} \int_{\mathbb{S}^1} f(r, \omega) \sinh r \, dr d\omega .
\end{align*}
The length of a curve
\begin{align*}
\gamma(t) = (\cosh r(t) ,\sinh r(t) \omega(t)),
\end{align*}
with $t$ varying from $a$ to $b$, is defined
\begin{align*}
L(\gamma) = \int_a^b \sqrt{\abs{\gamma' (t)}^2 + \abs{\sinh r(t)}^2 \abs{\omega' (t)}^2} \, dt .
\end{align*}

Let $\0 = \{(1,0,0) \}$ denote the origin of $\h^2$. The distance of a point to $\0$ is
\begin{align*}
d((\cosh r, \sinh r\omega) , \0) =r.
\end{align*}
More generally, the distance between two arbitrary points is
\begin{align*}
d(x, x') = \cosh^{-1} ([x , x']).
\end{align*}

The general definition of the Laplace-Beltrami operator is given by
\begin{align*}
\Delta_{\h^2} = \partial_r^2 + \frac{\cosh r}{\sinh r} \partial_r + \frac{1}{\sinh^2 r} \Delta_{\mathbb{S}^1} .
\end{align*}

\begin{rmk}
The form of the Laplace-Beltrami operator implies that there will be no scaling symmetry in $\h^2$ as we usually have in the $\R^d$ setting. 
\end{rmk}

\subsection{Tools on $\h^2$.} In this subsection we recall some important and classical analysis developed for the hyperbolic spaces.
\subsubsection{Fourier Transform on $\h^d$}
For $\theta \in \mathbb{S}^{d-1}$ and $\lambda$ a real number, the functions of the type
\begin{align*}
h_{\lambda,\theta} (x) = [x , \Lambda(\theta)]^{i\lambda -\frac{d-1}{2}},
\end{align*}
where $\Lambda(\theta)$ denotes the point of $\R^{d+1}$ given by $(1, \theta)$, are generalized eigenfunctions of the Laplacian-Beltrami operator. Indeed, we have
\begin{align*}
-\Delta_{\h^d} h_{\lambda, \theta} = \parenthese{\lambda^2 + \frac{(d-1)^2}{4}} h_{\lambda,\theta}.
\end{align*}
The Fourier transform on $\h^d$ is defined as 
\begin{align*}
\hat{f}(\lambda ,\theta) := \int_{\h^d} h_{\lambda , \theta} (x) f(x) \, d x,
\end{align*}
and the Fourier inversion formula on $\h^d$ takes the form of 
\begin{align*}
f(x) = \int_{-\infty}^{\infty} \int_{\mathbb{S}^{d-1}} \bar{h}_{\lambda ,\theta} (x) \hat{f}(\lambda ,\theta) \frac{d\theta d\lambda}{\abs{c(\lambda)}^2},
\end{align*}
where $c(\lambda)$ is the Harish-Chandra coefficient
\begin{align*}
\frac{1}{\abs{c(\lambda)}^2} = \frac{1}{2 (2 \pi)^d} \frac{\abs{\Gamma(i\lambda + \frac{d-1}{2})}^2}{\abs{\Gamma(i\lambda)}^2} .
\end{align*}

\subsubsection{Strichartz Estimates} In this subsection we recall the Strichartz estimates proved in the hyperbolic space.
We say that a couple $(q,r)$ is admissible if $(\frac{1}{q},\frac{1}{r})$ belong to the triangle $T_d = \{(\frac{1}{q},\frac{1}{r}) \in (0,\frac{1}{2}] \times ( 0,\frac{1}{2}) \, \big| \, \frac{2}{q} + \frac{d}{r} \geq \frac{d}{2} \} \cup \{ (0,\frac{1}{2})\}$. We have the following theorem.
\begin{thm}[Strichartz estimates in \cite{AP, IS}]
Assume $u$ is the solution to the inhomogeneous initial value problem
\begin{align}\label{InSh}
\begin{cases}
i \partial_t u + \Delta_{\h^d} u = F, & t \in \R , \quad x \in \h^d,\\
u(0,x) = \phi(x). & 
\end{cases}
\end{align}
Then, for any admissible exponents $(q,r)$ and $(\tilde{q}, \tilde{r})$ we have the Strichartz estimates:
\begin{align*}
\norm{u}_{L_t^q L_x^r (\R \times \h^d)} \lesssim \norm{\phi}_{L_x^2 (\h^d)} + \norm{F}_{L_t^{\tilde{q}'}  L_x^{\tilde{r}'} (\R \times \h^d)}.
\end{align*}
\end{thm}
\begin{rmk}
Strichartz estimates are better in $\h^d$ in the sense that the set $T_d$ of admissible pairs for $\h^d$ is much wider than the corresponding set $I_d$ for $\R^d$ in \eqref{Id} (which is just the lower edge of the triangle). See also Figure \ref{figure} below.
\end{rmk}

\begin{figure}[!htp]
	\centering
	\includegraphics[height=5cm,width=13cm]{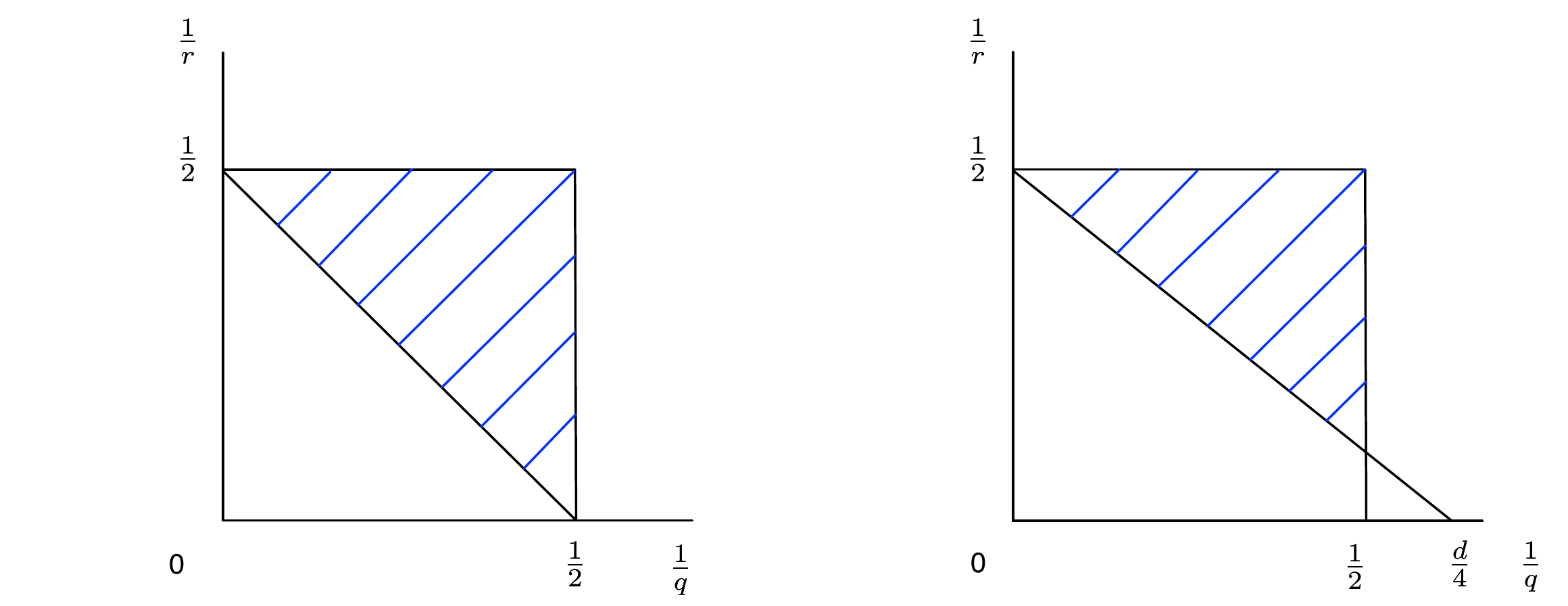}
	\caption{Strichartz admissible pair regions ($d=2$ and $d \geq 3$) for the hyperbolic space.}\label{figure}
\end{figure}

\begin{defn}[Strichartz spaces]\label{defn Strichartz spaces}
We define the Banach space
\begin{align*}
S^0 (I) = \bracket{f \in C(I : L^2(\h^2)) : \norm{f}_{S^0 (I)} = \sup_{(q,r) \text{ admissible }} \norm{f}_{L_t^q L_x^r (I  \times \h^2)} < \infty} .
\end{align*}
Also we define the Banach space $S^{\sigma} (I)$, where $\sigma >0$,
\begin{align*}
S^{\sigma} (I) = \bracket{f \in C(I : H^{\sigma} (\h^2)) : \norm{f}_{S^{\sigma}(I)} = \norm{(-\Delta)^{\frac{\sigma}{2}} f}_{S^0(I)} < \infty} .
\end{align*}
\end{defn}

\subsubsection{Local Smoothing Estimates in the Hypebolic Space}\label{lem H LSE}
\begin{thm}[Theorem 1.2 in \cite{Ka}: Local Smoothing Estimates in $\h^2$] For any $\varepsilon >0$,
\begin{align*}
& \norm{\inner{x}^{-\frac{1}{2} -\varepsilon} \abs{\nabla}^{\frac{1}{2}} e^{it\Delta} f}_{L_{t,x}^2 (\R \times \h^2)} \lesssim \norm{f}_{L_x^2(\h^2)} ,\\
& \norm{\inner{x}^{-\frac{1}{2} -\varepsilon} \nabla \int_0^t e^{i(t-s) \Delta} F(s,x) \, ds}_{L_{t,x}^2 (\R \times \h^2)} \lesssim \norm{\inner{x}^{\frac{1}{2} +\varepsilon}  F}_{L_{t,x}^2 (\R \times \h^2)} .
\end{align*}
\end{thm}
\begin{rmk}
In \cite{Ka}, the author considered more general manifolds that there are denoted with $X$. To obtain the theorem above one needs to take  $p(\lambda) = \abs{\lambda}^2$, $p(D) = -\Delta_{X} - \abs{\rho}^2$ and $m=2$. 
\end{rmk}

\subsubsection{Heat-Flow-Based Littlewood-Paley Projections and Functional Inequalities on $\h^2$}

The Littlewood-Paley projections on $\h^2$ that we use in this paper are based on the linear heat equation $e^{s\Delta}$. It turned out in fact that for us  this is a great substitute for the standard Littlewood-Paley decomposition used  in $\R^d$, since in $\h^d$ one cannot localize in frequencies efficiently. We report below several results that first appeared in \cite{LLOS}.
\begin{defn}[Section 2.7.1 in \cite{LLOS}: Heat-flow-based Littlewood-Paley projections]\label{defn LP}
For any $s> 0$, we define
\begin{align*}
P_{\geq s} f = e^{s\Delta} f , \quad P_s f = s(-\Delta) e^{s\Delta} f .
\end{align*}
By the fundamental theorem of calculus, it is straightforward to verify that
\begin{align*}
P_{\geq s} f = \int_s^{\infty} P_{s'} f \, \frac{ds'}{s'} \quad \text{ for } s >0.
\end{align*}
In particular, we have
\begin{align*}
f = \int_0^{\infty} P_{s'} f \, \frac{ds'}{s'} ,
\end{align*}
which is the basic identity that relates $f$ with its Littlewood-Paley resolution $\{P_s f\}_{s \in (0,\infty)}$. We also have
\begin{align*}
P_{\leq s} f = \int_0^s P_{s'} f \, \frac{ds'}{s'} .
\end{align*}
\end{defn}
\begin{rmk}
Intuitively, $P_s f$ may be interpreted as a projection of $f$ to frequencies comparable to $s^{-\frac{1}{2}}$. $P_{\geq s}$ and $P_{\leq s}$ can be viewed  as the projections into low and high frequencies, respectively.
\end{rmk}

\begin{lem}[Lemma 2.5 in \cite{LLOS}]\label{lem Ps bdd}
Let $1< p < \infty$ and $p \leq q \leq \infty$. Let $\rho_0$ satisfy
\begin{align*}
0 < \rho_0^2 < \frac{1}{2} \min \{\frac{1}{p} , 1- \frac{1}{p} \}.
\end{align*}
For $f \in L_x^p (\h^2)$ and $s > 0$, we have
\begin{align*}
\norm{e^{s\Delta} f}_{L_x^q (\h^2)} + \norm{s \Delta e^{s\Delta} f}_{L_x^q (\h^2)} \lesssim s^{-(\frac{1}{p} - \frac{1}{q})} e^{-\rho_0^2 s} \norm{f}_{L_x^p (\h^2)} . 
\end{align*}
\end{lem}

\begin{rmk}\label{rmk Ps bdd}
In particular, if $p=q$ in Lemma \ref{lem Ps bdd},
\begin{align*}
\norm{e^{s\Delta} f}_{L_x^p (\h^2)} + \norm{s \Delta e^{s\Delta} f}_{L_x^p (\h^2)} \lesssim  e^{-\rho_0^2 s} \norm{f}_{L_x^p (\h^2)} . 
\end{align*}
That is,
\begin{align*}
\norm{P_{\geq s} f}_{L_x^p (\h^2)} + \norm{P_s f}_{L_x^p (\h^2)} \lesssim  \norm{f}_{L_x^p (\h^2)} . 
\end{align*}
\end{rmk}

\begin{lem}[Corollary 2.7 in \cite{LLOS}]\label{lem Alpha bdd}
Let $0 < \alpha < 1$ and $1 < p < \infty$. For $f \in L^p (\h^2)$, we have
\begin{align*}
\norm{s^{\alpha} (-\Delta)^{\alpha} e^{s\Delta} f}_{L_x^p (\h^2) } \lesssim \norm{f}_{L_x^p (\h^2)} .
\end{align*}
\end{lem}

\begin{lem}[Lemma 2.9 in \cite{LLOS}: Boundedness of Riesz transform]
Let $f \in C_0^{\infty} (\h^2)$. Then for $1 < p < \infty$ we have
\begin{align*}
\norm{\nabla f}_{L_x^p(\h^2)} \simeq \norm{(-\Delta)^{\frac{1}{2}} f}_{L_x^p(\h^2)} . 
\end{align*}
\end{lem}

\begin{lem}[Lemma 2.10 in \cite{LLOS}: $L^p$ interpolation inequalities]\label{lem Interpolation}
Let $f \in C_0^{\infty} (\h^2)$. Then for any $0 \leq \beta \leq \alpha$ and $1 < p < \infty$ we have
\begin{align*}
\norm{(-\Delta)^{\beta} f}_{L_x^p (\h^2)} \lesssim \norm{f}_{L_x^p(\h^2)}^{1- \frac{\beta}{\alpha}} \norm{(-\Delta)^{\alpha} f}_{L_x^p(\h^2)}^{\frac{\beta}{\alpha}}.
\end{align*}
Moreover, for $1 < p < \infty$, $p \leq q \leq \infty$ and $0 < \theta = \frac{1}{\alpha} (\frac{1}{p} -\frac{1}{q}) < 1$, we have
\begin{align*}
\norm{f}_{L_x^q(\h^2)} \lesssim \norm{f}_{L_x^p(\h^2)}^{1-\theta} \norm{(-\Delta)^{\alpha} f}_{L_x^p(\h^2)}^{\theta}.
\end{align*}
\end{lem}

\begin{lem}[Sobolev embedding]\label{lem Sobolev}
\begin{align*}
W^{s,p} (\h^d) \hookrightarrow L^p (\h^d), \quad \text{if } 1 < p \leq q < \infty \text{ and } \frac{1}{p}-\frac{1}{q} =\frac{s}{d}.
\end{align*}
\end{lem}

\begin{lem}[Lemma 2.12 in \cite{LLOS}: Gagliardo-Nirenberg inequality]\label{lem GN}
Let $f \in C_0^{\infty} (\h^2)$. Then for any $1<p< \infty$, $p \leq q \leq \infty$ and $0< \theta < 1$ such that $\frac{1}{q} = \frac{1}{p} -\frac{\theta}{2}$, we have
\begin{align*}
\norm{f}_{L_x^q(\h^2)} \lesssim \norm{f}_{L_x^p(\h^2)}^{1-\theta} \norm{\nabla f}_{L_x^p(\h^2)}^{\theta}.
\end{align*}
In particular, for any $s>0$
\begin{align*}
\norm{f}_{L_x^{\infty}(\h^2)} \lesssim \norm{f}_{L_x^4(\h^2)}^{\frac{1}{2}} \norm{\nabla f}_{L_x^4(\h^2)}^{\frac{1}{2}}  \lesssim \norm{f}_{L_x^2(\h^2)}^{\frac{1}{4}} \norm{\nabla f}_{L_x^2(\h^2)}^{\frac{1}{2}} \norm{\Delta f}_{L_x^2(\h^2)}^{\frac{1}{4}} \lesssim s^{-\frac{1}{2}} \norm{\inner{s\Delta} f}_{L_x^2(\h^2)} .
\end{align*}
\end{lem}

\begin{lem}[Proposition 2.14 in \cite{LLOS}: Sobolev product rule]\label{lem Product rule}
For $\sigma >0$, we have
\begin{align*}
\norm{fg}_{H_x^{\sigma}(\h^2)} \lesssim \norm{f}_{L_x^{\infty}(\h^2)} \norm{g}_{H_x^{\sigma}(\h^2)} + \norm{f}_{H_x^{\sigma}(\h^2)} \norm{g}_{L_x^{\infty}(\h^2)} .
\end{align*}
\end{lem}

\begin{lem}[General Sobolev product rule]\label{lem G Product rule}
For $\sigma >0$, we have
\begin{align*}
\norm{fg}_{W_x^{\sigma, r }(\h^2)} \lesssim \norm{f}_{W_x^{\sigma, p_1}(\h^2)} \norm{g}_{L_x^{p_2}(\h^2)} + \norm{f}_{L_x^{q_1}(\h^2)} \norm{g}_{W_x^{\sigma,q_2}(\h^2)} .
\end{align*}
\end{lem}

\begin{rmk}
Lemma \ref{lem G Product rule} allows more possible $L^p$ norms than Lemma \ref{lem Product rule} in the product rule. The proof of Lemma \ref{lem G Product rule} is using Triebel's argument in \cite{Tr} (see for example Sections 7.2.2 and 7.2.4), and also can be found in Proposition 2.14 of \cite{LLOS}. This proof relies on a localization lemma (see for example Lemma 2.16 in \cite{LLOS}) to reduce to the standard Sobolev product rule.
\end{rmk}

\begin{lem}[Bernstein inequalities]\label{lem Bernstein}
For $0 \leq \beta < \alpha < \beta +1$
\begin{align*}
\norm{ (-\Delta)^{\beta} P_{\leq s} f}_{L_x^2(\h^2)} & \lesssim  s^{\alpha -\beta}\norm{(-\Delta)^{\alpha} f}_{L_x^2(\h^2)}  ,\\
\norm{(-\Delta)^{\alpha} P_{\geq s} f}_{L_x^2(\h^2)}  & \lesssim s^{\beta -\alpha} \norm{(-\Delta)^{\beta} f}_{L_x^2(\h^2)}  .
\end{align*}
\end{lem}

\begin{proof}
Using Definition \ref{defn LP}, Lemma \ref{lem Alpha bdd} and the formal property $(-\Delta)^{a}  (-\Delta)^{b} = (-\Delta)^{a+b} $,
\begin{align*}
\norm{ (-\Delta)^{\beta} P_{\leq s} f}_{L_x^2(\h^2)} & = \norm{\int_0^s (-\Delta)^{\beta} t (-\Delta) e^{t\Delta} f \, \frac{dt}{t}}_{L_x^2(\h^2)}  \lesssim \int_0^s \norm{ (-\Delta)^{\beta+1} e^{t\Delta} f }_{L_x^2(\h^2)} \, dt \\
& = \int_0^s t^{ \alpha -\beta-1} \norm{t^{1+\beta -\alpha} (-\Delta)^{1+\beta -\alpha} e^{t\Delta} (-\Delta)^{\alpha}  f}_{L_x^2(\h^2)}  \, dt\\
& \lesssim \int_0^s t^{\alpha -\beta-1} \norm{(-\Delta)^{\alpha} f}_{L_x^2(\h^2)}  \, dt =  \int_0^s t^{\alpha -\beta-1} \, dt \norm{(-\Delta)^{\alpha} f}_{L_x^2(\h^2)}  \\
& = s^{\alpha -\beta}\norm{(-\Delta)^{\alpha} f}_{L_x^2(\h^2)}  .
\end{align*}
Similarly,
\begin{align*}
\norm{(-\Delta)^{\alpha} P_{\geq s} f}_{L_x^2(\h^2)}  = s^{\beta-\alpha}  \norm{ s^{\alpha-\beta}(-\Delta)^{\alpha-\beta}e^{s\Delta}  (-\Delta)^{\beta} f}_{L_x^2(\h^2)}  \lesssim s^{\beta -\alpha} \norm{(-\Delta)^{\beta} f}_{L_x^2(\h^2)}  .
\end{align*}
\end{proof}

\subsubsection{Radial Sobolev Embeddings}

\begin{lem}[Lemma 2.13 in \cite{LLOS}: Radial Sobolev embeddings in $\h^2$]\label{lem H Radial Sobolev}
For any $s>0$ and any function $f$ radial,
\begin{align*}
\norm{\sinh^{\frac{1}{2}} (r) f}_{L_x^{\infty}(\h^2)} \lesssim \norm{f}_{L_x^2(\h^2)}^{\frac{1}{2}} \norm{\nabla f}_{L_x^2(\h^2)}^{\frac{1}{2}}  \lesssim s^{-\frac{1}{4}} \norm{\inner{s\Delta}^{\frac{1}{2}} f}_{L_x^2(\h^2)} .
\end{align*}
\end{lem}

\begin{cor}[Frequency localized radial Sobolev embeddings in $\h^2$]\label{cor H Freq radial Sobolev}
For any $s>0$ and $f$ radial,
\begin{align*}
\norm{\sinh^{\frac{1}{2}} (r) P_s f}_{L_x^{\infty}(\h^2)} \lesssim s^{-\frac{1}{4}} \norm{P_{\frac{s}{2}} f}_{L_x^2(\h^2)} .
\end{align*}
\end{cor}

\begin{proof}
Taking $f = P_{s} f$ in Lemma \ref{lem H Radial Sobolev}, we have 
\begin{align*}
\norm{\sinh^{\frac{1}{2}} (r) P_s f}_{L_x^{\infty}(\h^2)} & \lesssim  s^{-\frac{1}{4}} \norm{\inner{s\Delta}^{\frac{1}{2}} P_s f}_{L_x^2(\h^2)} = s^{-\frac{1}{4}}  \parenthese{\norm{P_s f}_{L_x^2(\h^2)}  + \norm{(s\Delta)^{\frac{1}{2}} P_s f}_{L_x^2(\h^2)} }.
\end{align*}
Both terms in parentheses are bounded by $\norm{P_{\frac{s}{2}} f}_{L_x^2(\h^2)}$. In fact, by Definition \ref{defn LP}, Remark \ref{rmk Ps bdd} and Lemma \ref{lem Alpha bdd},
\begin{align*}
& \norm{P_s f}_{L_x^2(\h^2)} = \norm{(s\Delta) e^{s\Delta} f}_{L_x^2(\h^2)}  = \norm{ e^{\frac{s}{2} \Delta} (s\Delta e^{\frac{s}{2}\Delta} f )}_{L_x^2(\h^2)} \lesssim \norm{P_{\frac{s}{2}} f}_{L_x^2(\h^2)},\\
& \norm{(s\Delta)^{\frac{1}{2}} P_s f}_{L_x^2(\h^2)} = \norm{(s\Delta)^{\frac{1}{2}} (s\Delta) e^{s\Delta} f}_{L_x^2(\h^2)} = \norm{(s\Delta)^{\frac{1}{2}} e^{\frac{s}{2} \Delta} (s\Delta e^{\frac{s}{2}\Delta} f )}_{L_x^2(\h^2)} \lesssim \norm{P_{\frac{s}{2}} f}_{L_x^2(\h^2)}.
\end{align*}
\end{proof}

\begin{cor}\label{cor H Radial Sobolev}
For any $s>0$, $\frac{1}{4} < \alpha < 1$ and $f$ radial, 
\begin{align*}
\norm{\sinh^{\frac{1}{2}} (r) f}_{L_x^{\infty}(\h^2)} \lesssim \norm{f}_{L_x^2(\h^2)}^{1-\frac{1}{4\alpha}} \norm{(-\Delta)^{\alpha} f}_{L_x^2(\h^2)}^{\frac{1}{4\alpha}}  .
\end{align*}
\end{cor}

\begin{proof}
We write $f$ into its Littlewood-Paley decomposition and use the frequency localized radial Sobolev embedding Corollary \ref{cor H Freq radial Sobolev}, then we have
\begin{align*}
\abs{\sinh^{\frac{1}{2}} (r) f} & = \abs{\sinh^{\frac{1}{2}} (r)  \int_0^{\infty}  P_s f \, \frac{ds}{s} }= \abs{\int_0^{\infty} \sinh^{\frac{1}{2}} (r)  P_s f \, \frac{ds}{s} } \\
& \lesssim \int_0^{\infty} s^{-\frac{1}{4}} \norm{P_{\frac{s}{2}} f}_{L_x^2(\h^2)} \, \frac{ds}{s}  \\
& = \int_0^T s^{-\frac{1}{4}} \norm{P_{\frac{s}{2}} f}_{L_x^2(\h^2)} \, \frac{ds}{s} +  \int_T^{\infty}  s^{-\frac{1}{4}} \norm{P_{\frac{s}{2}} f}_{L_x^2(\h^2)} \, \frac{ds}{s} : = I + II.
\end{align*}
Here $T$ is a constant that will be chosen later. Before calculating $I$ and $II$, let us recall the following two estimates for $\norm{P_s f}_{L_x^p(\h^2)}$: for $1<p<\infty$ and $0 < \theta < 1$,
\begin{align}
\norm{P_s f}_{L_x^p(\h^2)} & = \norm{s\Delta e^{s\Delta} f}_{L_x^p(\h^2)} \lesssim \norm{f}_{L_x^p (\h^2)}, \label{eq P_s f 1}\\
\norm{P_s f}_{L_x^p(\h^2)} & = \norm{s\Delta e^{s\Delta} f}_{L_x^p(\h^2)} = \norm{s^{\theta} s^{1-\theta} (-\Delta)^{1-\theta} e^{s\Delta} (-\Delta)^{\theta} f}_{L_x^p(\h^2)}  \lesssim s^{\theta} \norm{(-\Delta)^{\theta} f}_{L_x^p(\h^2)} \label{eq P_s f 2}.
\end{align}
Now by \eqref{eq P_s f 2}, for $\frac{1}{4} < \alpha < 1$
\begin{align*}
I & = \int_0^T s^{-\frac{1}{4}} \norm{P_{\frac{s}{2}} f}_{L_x^2(\h^2)} \, \frac{ds}{s}  \lesssim \int_0^T s^{-\frac{1}{4}}  s^{\alpha} \norm{(-\Delta)^{\alpha} f}_{L_x^2 (\h^2)} \,  \frac{ds}{s} \\
& = (\int_0^T s^{\alpha-\frac{1}{4}}  \, \frac{ds}{s}) \norm{(-\Delta)^{\alpha} f}_{L_x^2 (\h^2)}  = T^{\alpha-\frac{1}{4}} \norm{(-\Delta)^{\alpha} f}_{L_x^2 (\h^2)}.
\end{align*}
By \eqref{eq P_s f 1}
\begin{align*}
II = \int_T^{\infty}  s^{-\frac{1}{4}} \norm{P_{\frac{s}{2}} f}_{L_x^2(\h^2)} \, \frac{ds}{s}  \lesssim (\int_T^{\infty} s^{-\frac{1}{4}} \, \frac{ds}{s}) \norm{f}_{L_x^2 (\h^2)} = T^{-\frac{1}{4}} \norm{f}_{L_x^2 (\h^2)}.
\end{align*}
Therefore, for any $T>0$
\begin{align*}
\norm{\sinh^{\frac{1}{2}} (r)  f}_{L_x^{\infty}(\h^2)}  & \lesssim T^{\alpha-\frac{1}{4}} \norm{(-\Delta)^{\alpha} f}_{L_x^2 (\h^2)}  + T^{-\frac{1}{4}} \norm{f}_{L_x^2 (\h^2)} .
\end{align*}
Optimizing the choice of $T$, we obtain
\begin{align*}
\norm{\sinh^{\frac{1}{2}} (r)  f}_{L_x^{\infty}(\h^2)}  & \lesssim \norm{(-\Delta)^{\alpha} f}_{L_x^2 (\h^2)}^{\frac{1}{4\alpha}}  \norm{f}_{L_x^2 (\h^2)}^{1-\frac{1}{4\alpha}}.
\end{align*}
\end{proof}

\section{Energy increment on $\h^2$}\label{sec Energy increment}

In this section we analyze a certain  energy increment. As mentioned in the introduction we  present a modified Morawetz type  estimate in Section 5, and in  Section 6 we conclude the global well-posedness and scattering proof by showing that  the space-time $L^4$ norm of the solution is uniformly bounded.

Let us recall schematically below the heat-flow-based Littlewood-Paley projections:
\begin{align*}
P_s f = s (-\Delta) e^{s\Delta} f & \leadsto \text{ a projection to frequencies comparable to $s^{-\frac{1}{2}}$} ,\\
P_{\geq s} f = e^{s\Delta} f = \int_s^{\infty} P_{s'} f \, \frac{ds'}{s'} & \leadsto \text{ a projection to frequencies lower than $\sim s^{-\frac{1}{2}}$} ,\\
P_{\leq s} f = \int_0^s P_{s'} f \, \frac{ds'}{s'} &  \leadsto \text{ a projection to frequencies higher than $\sim s^{-\frac{1}{2}}$} .
\end{align*}
Now we decompose the initial data $\phi$ into a low frequency component $\eta_0= P_{> s_0} \phi$ and a high frequency component  $\psi_0= P_{\leq s_0} \phi$, where $s_0^{-1}$ is a fixed large frequency and will be determined later in the proof. Note that $s_0^{-\frac{1}{2}}$ plays the same role as $N_0$ in \cite{B1}.

Using the decomposition above, we would like to write $u$ into the sum of the following two solutions $\psi$ and $\zeta$, where $\psi = e^{it\Delta} P_{\leq s_0} \phi $ solves the linear Schr\"odinger with high frequency data
\begin{align}\label{eq psi}
\begin{cases}
i \partial_t \psi + \Delta_{\h^2} \psi= 0, \\
\psi(0,x) =\psi_0 = P_{\leq s_0} \phi ,
\end{cases}
\end{align}
and $\zeta$ solves the difference equation with low frequency data
\begin{align}\label{eq zeta}
\begin{cases}
i \partial_t \zeta + \Delta_{\h^2} \zeta = \abs{u}^2u =G(\zeta, \psi) , \\
\zeta(0,x) = \eta_0 = P_{> s_0} \phi ,
\end{cases}
\end{align}
here $G(\zeta, \psi) = \abs{\zeta+\psi}^2(\zeta+\psi) = \abs{\zeta}^2 \zeta + \OO(\zeta^2 \psi) +\OO(\zeta \psi^2) + \OO(\psi^3)$.

\subsection{Main results in the section}

The main results in this section are a local energy increment (Proposition \ref{prop Local}) and a conditional global energy increment (Proposition \ref{prop Global}) for the solution $\zeta$. 
\begin{prop}[Local energy increment]\label{prop Local}
Consider $u$ as in \eqref{NLS} defined on $I \times \h^2$ where $I = [0,\tau]$, such that
\begin{align}\label{eq Asm L^4 small}
\norm{u}_{L_{t,x}^4 (I \times \h^2)}^4 = \varepsilon
\end{align}
for some universal constant $\varepsilon$. Then for $s > \frac{3}{4}$ and sufficiently small $s_0$, the solution $\zeta$, under the decomposition $u = \psi + \zeta$ defined as in \eqref{eq psi} and \eqref{eq zeta}, satisfies the following energy increment
\begin{align*}
E(\zeta (\tau) ) \leq E(\zeta(0)) + C s_0^{\frac{3}{2}s -\frac{5}{4}}  .
\end{align*}
\end{prop}

\begin{prop}[Conditional global energy increment]\label{prop Global}
Consider $u$ as in \eqref{NLS} defined on $[0, T] \times \h^2$ where
\begin{align*}
\norm{u}_{L_{t,x}^4 ([0,T] \times \h^2)}^4 \leq M
\end{align*}
for some constant $M$. Then for $s> \frac{3}{4}$ and sufficiently small $s_0$, the energy of $\zeta$ satisfies the following energy increment
\begin{align*}
E(\zeta (T) ) \leq E(\zeta(0)) + C \frac{M}{\varepsilon}  s_0^{\frac{3}{2}s -\frac{5}{4}}  .
\end{align*}
where $\varepsilon$ is the small constant in Proposition \ref{prop Local}.
\end{prop}

\begin{rmk}
$T$ could be infinity. In fact, the ultimate goal of this paper is to show that the spacetime $L^4$ norm is bounded for all time intervals, which implies scattering. 
\end{rmk}

\subsection{Proof of Proposition \ref{prop Local}}
To analyze the behavior of the solution $\zeta$ more carefully, we first make a further decomposition. That is, we would like to separate the differential equation \eqref{eq zeta} into a cubic NLS with low frequency data,
\begin{align}\label{eq zeta1}
\begin{cases}
i \partial_t \zeta_1 + \Delta_{\h^2} \zeta_1 = \abs{\zeta_1}^2 \zeta_1 , \\
\zeta_1(0,x) = \eta_0 = P_{> s_0} \phi ,
\end{cases}
\end{align}
and a difference equation with zero initial value,
\begin{align}\label{eq zeta2}
\begin{cases}
i \partial_t \zeta_2 + \Delta_{\h^2} \zeta_2 = \abs{u}^2 u - \abs{\zeta_1}^2 \zeta_1 , \\
\zeta_2(0,x) =0.
\end{cases}
\end{align}
Hence $\zeta = \zeta_1 + \zeta_2$ and the full solution $u$ is the sum of these three solutions:
\begin{align}\label{eq u sum}
u = \zeta_1 + \zeta_2 + \psi .
\end{align}

It is worth mentioning that the  decomposition in Bourgain's work \cite{B1} is a cubic NLS with low frequency data,
\begin{align}\label{eq Low Freq}
\begin{cases}
i \partial_t u_0 + \Delta u_0 = \abs{u_0}^2 u_0, \\
u_0(0,x) = \phi_0(x) =P_{<N_0} \phi. & 
\end{cases}
\end{align}
and a difference equation with high frequency data
\begin{align}\label{eq High Freq}
\begin{cases}
i \partial_t v + \Delta v  = F(u_0, v), \\
v(0,x) = \psi_0(x) =P_{\geq N_0} \phi,
\end{cases}
\end{align}
where $F(u_0,v) = \abs{v}^2 v + 2u_0 \abs{v}^2 + \bar{u}_0 v^2 + u_0^2 \bar{v} + 2 \abs{u_0}^2 v $. 
Then the full solution $u$ is $u = u_0 +v$.  In our work we need to be more careful.

Notice that $s_0^{-\frac{1}{2}}$ plays the same role as $N_0$ in \cite{B1}. When comparing these two decomposition, we can relate them in the following sense: $\zeta_1 $ is the same as $u_0$ and $\psi + \zeta_2$ is in fact $v$, where $\psi$ is the linear solution in \eqref{eq High Freq} and $\zeta_2$ is the Duhamel term $w$ in \eqref{eq High Freq}.
\begin{align*}
\zeta_1  \leftrightsquigarrow u_0 \qquad 
\zeta_2   \leftrightsquigarrow w \qquad 
\psi  \leftrightsquigarrow e^{it\Delta} P_{\leq s_0} \phi.
\end{align*}

\noindent {\bf Step 1: Understanding the decomposed initial data.}
Recall that we decomposed the initial data $\phi = \eta_0 + \psi_0$, where $\eta_0 = P_{> s_0} \phi$ and $\psi_0 = P_{\leq s_0} \phi$. Here we list several facts of the decomposed initial data $\eta_0$ and $\psi_0$.
\begin{fact}\label{fact eta0}
For the low frequency data $\eta_0$,
\begin{enumerate}[\it (1)]
\item
$\eta_0 \in H^1$, and $\norm{\eta_0 }_{H_x^1 (\h^2)} \lesssim s_0^{\frac{1}{2}(s -1)} $,

\item
$\norm{\eta_0 }_{H_x^{\sigma} (\h^2)} \lesssim s_0^{\frac{\sigma}{2}(s-1)}$ for $0 < \sigma <1$,

\item
$E(\eta_0) \lesssim  s_0^{s-1}$.
\end{enumerate}
\end{fact}

In fact, by Bernstein inequality (Lemma \ref{lem Bernstein}),
\begin{align*}
\norm{\eta_0 }_{H_x^1 (\h^2)} = \norm{P_{> s_0} \phi}_{H_x^1 (\h^2)}  \lesssim s_0^{\frac{1}{2}(s -1)}  \norm{P_{> s_0} \phi}_{H_x^s (\h^2)} \lesssim s_0^{\frac{1}{2}(s -1)} .
\end{align*}
This gives {\it (1)}. {\it (2)} follows from Lemma \ref{lem Interpolation} by interpolating $L^2$ and $H^1$ norms,
\begin{align*}
\norm{\eta_0}_{\dot{H}_x^{\sigma}(\h^2)} & \lesssim \norm{\eta_0}_{L_x^2(\h^2)}^{1-\sigma} \norm{\eta_0}_{H_x^1(\h^2)}^{\sigma} \lesssim s_0^{ \frac{\sigma}{2} (s-1)}.
\end{align*}
Then by Sobolev embedding (Lemma \ref{lem Sobolev}) and {\it (2)}, we see that 
\begin{align*}
E(\eta_0) \lesssim \norm{\eta_0}_{\dot{H}_x^1(\h^2)}^2 +\norm{\eta_0}_{L_x^4(\h^2)}^4 \lesssim  \norm{\eta_0}_{\dot{H}_x^1(\h^2)}^2  + \norm{\eta_0}_{H_x^{\frac{1}{2}}(\h^2)}^4 \lesssim s_0^{s-1}  .
\end{align*}

\begin{fact}\label{fact psi0}
For the high frequency data $\psi_0$,
\begin{enumerate}[\it (1)]
\item
$\norm{\psi_0}_{L_x^2(\h^2)} \lesssim s_0^{\frac{1}{2}s}$,

\item
$\norm{\psi_0}_{H_x^s(\h^2)} \lesssim 1$.
\end{enumerate}
\end{fact}

 Here {\it (1)} follows from Bernstein inequality (Lemma \ref{lem Bernstein}),
\begin{align*}
\norm{\psi_0}_{L_x^2(\h^2)} = \norm{P_{\leq s_0} \phi}_{L_x^2(\h^2)}  \lesssim s_0^{\frac{1}{2}s} \norm{\phi}_{H_x^s(\h^2)} \lesssim s_0^{\frac{1}{2}s},
\end{align*}
while {\it (2)} is due to the fact that $\phi$ being in $H^s$.

\noindent {\bf Step 2: Estimation on the solution $\psi$ of \eqref{eq psi}.}

In fact, the solution $\psi$ of the linear equation \eqref{eq psi} is global, although it lives in a rough space $H^s$. Moreover, from  the linear Strichartz estimates, Lemma \ref{lem Bernstein} and {\it (1)} in Fact \ref{fact psi0} one has
\begin{align}\label{eq psi L^4}
\norm{\psi}_{L_{t,x}^4(\R \times \h^2)} = \norm{e^{it\Delta} \psi_0}_{L_{t,x}^4(\R \times \h^2)}\lesssim \norm{\psi_0}_{L_x^2(\h^2)}  \lesssim s_0^{\frac{1}{2}s}.
\end{align}
More generally,
\begin{align}\label{eq psi S^sigma}
\norm{\psi}_{S^0(\R)} \lesssim s_0^{\frac{1}{2}s} \text{ and } \norm{\psi}_{S^{\sigma}(\R)} \lesssim s_0^{\frac{1}{2}(s-\sigma)} \text{ for } 0 \leq \sigma \leq s. 
\end{align}

\noindent {\bf Step 3: Estimation on the solution $\zeta_1$ of \eqref{eq zeta1}.}

\begin{lem}\label{lem Est zeta1}
Due to the low frequency component $\eta_0(x)$ of $\phi$ being in $H^1$, $\zeta_1(t)$ is a global solution and $\norm{\zeta_1 (t)}_{H_x^1 (\h^2)}$ is conserved. More precisely, 
\begin{enumerate}[\it (1)]
\item
$\zeta_1 (t)$ exists globally, and $E(\zeta_1 )(t) = E(\eta_0) \lesssim s_0^{s-1}$,

\item
$\norm{\zeta_1}_{L_{t,x}^4 (I \times \h^2)}^4 \lesssim \varepsilon+ s_0^{2s}  + \norm{\zeta_2}_{L_{t,x}^4 (I \times \h^2)}^4$
\end{enumerate}
\end{lem}

\begin{rmk}
Ultimately, we will show $\norm{\zeta_1}_{L_{t,x}^4 (I \times \h^2)}^4 \lesssim \varepsilon$ in Corollary \ref{cor Est zeta1}, and here {\it (2)} is an intermediate step.
\end{rmk}

\begin{proof}[Proof of Lemma \ref{lem Est zeta1}]
First, with the conservation of $E(\zeta_1)$ and {\it (3)} in Fact \ref{fact eta0}, it is easy to see that
\begin{align*}
E(\zeta_1(t)) \equiv E(\eta_0) \lesssim s_0^{s-1}.
\end{align*}
With the initial data $\eta_0$ being in $H^1$, thanks to \cite{IS}, $\zeta_1$ is globally well-posed, which proves {\it (1)}.

For {\it (2)}, recall the decomposition of $u$ in \eqref{eq u sum}, then we simply use the triangle inequality, the assumption \eqref{eq Asm L^4 small}, and \eqref{eq psi L^4} and obtain
\begin{align*}
\norm{\zeta_1}_{L_{t,x}^4 (I \times \h^2)}^4 \lesssim \norm{u}_{L_{t,x}^4 (I \times \h^2)}^4 + \norm{\psi}_{L_{t,x}^4 (I \times \h^2)}^4 + \norm{\zeta_2}_{L_{t,x}^4 (I \times \h^2)}^4  \lesssim \varepsilon+ s_0^{2s}  + \norm{\zeta_2}_{L_{t,x}^4 (I \times \h^2)}^4 .
\end{align*}

\end{proof}

\noindent {\bf Step 4: Estimation on the solution $\zeta_2$ of \eqref{eq zeta2} and extra estimates on $\zeta_1$.}

Recall \eqref{eq zeta2}
\begin{align*}
\begin{cases}
i \partial_t \zeta_2 + \Delta_{\h^2} \zeta_2 = F(\zeta_1, \zeta_2 , \psi), \\
\zeta_2(0,x) =0,
\end{cases}
\end{align*}
where 
\begin{align*}
F(\zeta_1, \zeta_2 , \psi) & = \abs{u}^2 u - \abs{\zeta_1}^2 \zeta_1 =  \OO( (\zeta_1 + \zeta_2 + \psi)^3)  -  \abs{\zeta_1}^2 \zeta_1 \\
& = \OO( \zeta_2^3) +  \OO( \psi^3) + \OO (\zeta_2^2 \zeta_1) +  \OO ( \psi^2 \zeta_1)  + \OO (\zeta_2 \zeta_1^2) + \OO (\psi \zeta_1^2). 
\end{align*}

\begin{lem}\label{lem Est zeta2}
The solution $\zeta_2$ satisfies the following estimates on $I$:
\begin{enumerate}
\item
$\norm{\zeta_2}_{L_{t,x}^4 (I \times \h^2)} \lesssim s_0^{\frac{1}{2}s}$, 

\item
$\norm{\zeta_2}_{L_t^{\infty} L_x^2(I \times \h^2)} \lesssim s_0^{\frac{1}{2}s}$,

\item
$\norm{\zeta_2}_{L_t^{\infty} H_x^1 (I \times \h^2)} \lesssim  s_0^{s -\frac{3}{4}} $.
\end{enumerate}
\end{lem}

\begin{proof}[Proof of Lemma \ref{lem Est zeta2}]
Noticing that \eqref{eq zeta2} has zero initial value, we write out the integral equation using its Duhamel formula
\begin{align*}
\zeta_2(t) = i \int_0^t e^{i(t-s) \Delta_{\h^2}} F \, ds.
\end{align*}
By Strichartz estimates and H\"older inequality, we have
\begin{align}\label{eq zeta2 L^4}
\begin{aligned}
\norm{\zeta_2}_{L_{t,x}^4 (I \times \h^2)} & \lesssim \norm{F}_{L_{t,x}^{\frac{4}{3}} (I \times \h^2)}  \lesssim \norm{\zeta_2}_{L_{t,x}^4 (I \times \h^2)}^3 + \norm{\psi}_{L_{t,x}^4 (I \times \h^2)}^3 \\
& + \norm{\zeta_2}_{L_{t,x}^4 (I \times \h^2)} \norm{\zeta_1}_{L_{t,x}^4 (I \times \h^2)}^2 + \norm{\psi}_{L_{t,x}^4 (I \times \h^2)} \norm{\zeta_1}_{L_{t,x}^4 (I \times \h^2)}^2 .
\end{aligned}
\end{align}
Note that here there should be two more nonlinear terms in $F$ that contribute to $\norm{F}_{L_{t,x}^{\frac{4}{3}}} $ in \eqref{eq zeta2 L^4}, which are $ \OO(\zeta_2^2 \zeta_1) $ and $\OO( \psi^2 \zeta_1)$. But we dropped them, since their contributions are controlled by a multiple of those of the four nonlinear terms that are written in \eqref{eq zeta2 L^4}. We will also drop them in the rest of the paper.

Using \eqref{eq psi L^4} and {\it (2)} in Lemma \ref{lem Est zeta1}, we write \eqref{eq zeta2 L^4} into
\begin{align*}
\begin{aligned}
\norm{\zeta_2}_{L_{t,x}^4 (I \times \h^2)} & \lesssim \norm{\zeta_2}_{L_{t,x}^4 (I \times \h^2)}^3 + s_0^{\frac{3}{2}s} + \norm{\zeta_2}_{L_{t,x}^4 (I \times \h^2)} (\varepsilon^{\frac{1}{2}}+ s_0^{s} + \norm{\zeta_2}_{L_{t,x}^4 (I \times \h^2)}^2) \\
& \quad + s_0^{\frac{1}{2}s}(\varepsilon^{\frac{1}{2}}+ s_0^{s}  + \norm{\zeta_2}_{L_{t,x}^4 (I \times \h^2)}^2) .
\end{aligned}
\end{align*}
Noticing that initially $\zeta_2(0) =0$, then by a continuity argument, we obtain
\begin{align*}
\norm{\zeta_2}_{L_{t,x}^4 (I \times \h^2)}  \lesssim s_0^{\frac{1}{2}s} ,
\end{align*}
which proves {\it (1)}. The estimate in \eqref{eq zeta2 L^4} also works for $\norm{\zeta_2}_{L_t^{\infty} L_x^2 (I \times \h^2)} $, hence {\it (2)} holds. 

We postpone the proof of {\it (3)} to Step 6.

\end{proof}

With enough estimates on $\zeta_2$ in hand, as we promised in Lemma \ref{lem Est zeta1}, we will finish the analysis of $\zeta_1$. 
\begin{cor}\label{cor Est zeta1}
As a consequence of {\it (1)} in Lemma \ref{lem Est zeta2}, 
\begin{enumerate}
\item
we improve the bound in Lemma \ref{lem Est zeta1} by $\norm{\zeta_1}_{L_{t,x}^4 (I \times \h^2)}^4 \lesssim \varepsilon +  s_0^{2s} \lesssim \varepsilon$, 

\item
and obtain $\norm{\zeta_1}_{S^{\sigma} (I )} \lesssim s_0^{\frac{\sigma}{2}(s-1)}$ where $0 \leq \sigma \leq 1$.
\end{enumerate}
\end{cor}

\begin{proof}[Proof of Corollary \ref{cor Est zeta1}]
Combining {\it (2)} in Lemma \ref{lem Est zeta1} and  {\it (1)} in Lemma \ref{lem Est zeta2}, it is easy to see that {\it (1)} holds. 

For {\it (2)}, we use the integral equation corresponding to the initial value problem via the Duhamel principle, Strichartz estimates and {\it (1)}, and we obtain
\begin{align*}
\norm{\zeta_1}_{S^0 (I )} & \lesssim \norm{\eta_0}_{L_x^2 (\h^2)} + \norm{\abs{\zeta_1}^2 \zeta_1}_{L_{t,x}^{\frac{4}{3}} (I \times \h^2)} \lesssim 1+ \norm{ \zeta_1}_{L_{t,x}^{4} (I \times \h^2)}^3  \lesssim 1+ \varepsilon^{\frac{3}{4}} ,\\
\norm{\zeta_1}_{S^1 (I )} & \lesssim \norm{\nabla \eta_0}_{L_x^2 (\h^2)} + \norm{\nabla \abs{\zeta_1}^2 \zeta_1}_{L_{t,x}^{\frac{4}{3}} (I \times \h^2)}  \lesssim s_0^{\frac{1}{2}(s-1)} + \norm{\zeta_1}_{S^1 (I )} \varepsilon^{\frac{1}{2}} .
\end{align*}
Therefore
\begin{align}\label{eq zeta1 S^0}
\norm{\zeta_1}_{S^0 (I )} \lesssim 1 ,
\end{align}
and
\begin{align}\label{eq zeta1 S^1}
\norm{\zeta_1}_{S^1 (I )} \lesssim s_0^{\frac{1}{2}(s-1)}.
\end{align}
Then interpolating \eqref{eq zeta1 S^0} and \eqref{eq zeta1 S^1}, we have that for $0 < \sigma < 1$
\begin{align*}
\norm{\zeta_1}_{S^{\sigma} (I )} \lesssim s_0^{\frac{\sigma}{2}(s-1)}.
\end{align*}
\end{proof}

\noindent {\bf Step 5: Local energy increment.}

Now we are ready to compute the energy increment from 0 to $\tau$ and show such increment is as described in Proposition \ref{prop Local}. That is, we will show 
\begin{align*}
E(\zeta (\tau) ) & = E(\zeta_1 (\tau) + \zeta_2 (\tau) )\\ 
&= E(\zeta_1 (\tau) ) + \parenthese{E (\zeta_1 (\tau) + \zeta_2 (\tau) ) -E(\zeta_1 (\tau) )} \\
& \leq E(\eta_0) + C s_0^{\frac{3}{2}s -\frac{5}{4}}  .
\end{align*}

In fact, a direct computation of the difference of the energy gives
\begin{align*}
& \quad \abs{E(\zeta_1(\tau) +  \zeta_2(\tau)) -E(\zeta_1 (\tau))} \\
& \lesssim \parenthese{\norm{\zeta_1(\tau)}_{H_x^1(\h^2)} + \norm{\zeta_2(\tau)}_{H_x^1(\h^2)}} \norm{\zeta_2(\tau)}_{H_x^1(\h^2)} + \norm{\parenthese{\abs{\zeta_1(\tau)} + \abs{\zeta_2(\tau)}}^3 \abs{\zeta_2(\tau)}}_{L_x^1(\h^2)}\\
&: = I + II .
\end{align*}
By the energy conservation of $\zeta_1$, {\it (1)} in Lemma \ref{lem Est zeta1} and {\it (3)} in Lemma \ref{lem Est zeta2}
\begin{align*}
I & = \parenthese{\norm{\zeta_1(\tau)}_{H_x^1(\h^2)} + \norm{\zeta_2(\tau)}_{H_x^1(\h^2)}} \norm{\zeta_2(\tau)}_{H_x^1(\h^2)} \\
& \leq E(\eta_0)^{\frac{1}{2}} \norm{\zeta_2(\tau)}_{H_x^1(\h^2)} + \norm{\zeta_2(\tau)}_{H_x^1(\h^2)}^2\\
& \lesssim s_0^{\frac{1}{2}(s-1)} \cdot s_0^{s-\frac{3}{4}} + s_0^{2(s-\frac{3}{4})} \\
& \lesssim s_0^{\frac{3}{2}s -\frac{5}{4}} , 
\end{align*}
for $s > \frac{1}{2}$. ($\norm{\zeta_1(\tau)}_{H_x^1} \norm{\zeta_2(\tau)}_{H_x^1}$ dominates in $I$.)

Using Sobolev embedding, Lemma \ref{lem Interpolation}, {\it (2), (3)} in Lemma \ref{lem Est zeta2} and {\it (1)} in Lemma \ref{lem Est zeta1}, we have the following $L^4$ norm estimates for $\zeta_1$ and $\zeta_2$
\begin{align*}
\norm{\zeta_2(\tau)}_{L_x^4(\h^2)} &\lesssim \norm{\zeta_2(\tau)}_{H_x^{\frac{1}{2}}(\h^2)} \lesssim \norm{\zeta_2(\tau)}_{L_x^2(\h^2)}^{\frac{1}{2}} \norm{\zeta_2(\tau)}_{H_x^1(\h^2)}^{\frac{1}{2}} \lesssim s_0^{\frac{3}{4}s -\frac{3}{8}}  ,\\
\norm{\zeta_1(\tau)}_{L_x^4(\h^2)} &\lesssim E(\zeta_1)^{\frac{1}{4}} \lesssim s_0^{\frac{1}{4} (s-1)} .
\end{align*}
Combining with  H\"older inequality, we compute
\begin{align*}
II & = \norm{\parenthese{\abs{\zeta_1(\tau)} + \abs{\zeta_2(\tau)}}^3 \abs{\zeta_2(\tau)}}_{L_x^1(\h^2)}\\
& \lesssim \parenthese{\norm{\zeta_1 (\tau)}_{L_x^4(\h^2)} + \norm{\zeta_2(\tau)}_{L_x^4(\h^2)} }^3 \norm{\zeta_2(\tau)}_{L_x^4(\h^2)}\\
& \lesssim (s_0^{\frac{1}{4}(s-1)} + s_0^{\frac{3}{4}s -\frac{3}{8}}  )^3 \cdot s_0^{\frac{3}{4}s -\frac{3}{8}} \\
& \lesssim s_0^{\frac{3}{2}s -\frac{9}{8}},
\end{align*}
for $s > \frac{1}{4}$.
($\norm{\zeta_1 (\tau)}_{L_x^4(\h^2)}^3  \norm{\zeta_2(\tau)}_{L_x^4(\h^2)}$ dominates in $II$, and $I$ dominates in $E(\zeta(\tau)) -E(\zeta(0))$.)

Now we finish the calculation of the analysis of the energy increment in Proposition \ref{prop Local}.

\noindent {\bf Step 6: Proof of {\it (3)} in Lemma \ref{lem Est zeta2}.}

Before proving {\it (3)} in Lemma \ref{lem Est zeta2}, we first state the following lemma,
\begin{lem}\label{lem zeta2 S^sigma}
For $t \in I $ defined in \eqref{eq Asm L^4 small} we have for $0 \leq \sigma \leq s$ 
\begin{align}
\norm{\zeta_2}_{S^{\sigma}(I)} \lesssim s_0^{\frac{1}{2} (\sigma s + s -\sigma)}  .
\end{align}
\end{lem}

\begin{proof}[Proof of Lemma \ref{lem zeta2 S^sigma}]
For $0  < \sigma < s$, by the integral equation, Bernstein inequality (Lemma \ref{lem Bernstein}), Lemma \ref{lem G Product rule}, Strichartz inequalities and \eqref{eq Asm L^4 small}
\begin{align}\label{eq zeta2 S^sigma}
\norm{\zeta_2}_{S^{\sigma}(I)} & \lesssim  \norm{\inner{-\Delta}^{\frac{\sigma}{2}}\OO(\zeta_2^3)}_{L_{t,x}^{\frac{4}{3}} (I \times \h^2)} + \norm{\inner{-\Delta}^{\frac{\sigma}{2}}\OO(\psi^3)}_{L_{t,x}^{\frac{4}{3}} (I \times \h^2)}  \notag\\
& \quad + \norm{\inner{-\Delta}^{\frac{\sigma}{2}}\OO (\zeta_2 \zeta_1^2)}_{L_{t,x}^{\frac{4}{3}} (I \times \h^2)}  + \norm{\inner{-\Delta}^{\frac{\sigma}{2}}\OO (\psi \zeta_1^2)}_{L_{t,x}^{\frac{4}{3}} (I \times \h^2)} \notag\\
& \lesssim  \norm{\zeta_2}_{S^{\sigma}(I)}  \norm{ \zeta_2}_{L_{t,x}^4 (I \times \h^2)}^2  + \norm{\psi}_{S^{\sigma}(I)} \norm{\psi}_{L_{t,x}^4 (I \times \h^2)}^2 \notag\\
& \quad  + \norm{\zeta_2}_{S^{\sigma}(I)}  \norm{\zeta_1}_{L_{t,x}^4 (I \times \h^2)}^2 + \norm{\zeta_1}_{S^{\sigma}(I)} \norm{\zeta_1}_{L_{t,x}^4 (I \times \h^2)} \norm{\zeta_2}_{L_{t,x}^4 (I \times \h^2)} \notag \\
& \quad  + \norm{\zeta_2}_{S^{\sigma}(I)}  \norm{\zeta_1}_{L_{t,x}^4 (I \times \h^2)}^2 + \norm{\zeta_1}_{S^{\sigma}(I)}\norm{\zeta_1}_{L_{t,x}^4 (I \times \h^2)} \norm{\psi}_{L_{t,x}^4 (I \times \h^2)} .
\end{align}
Using {\it (1)} in Lemma \ref{lem Est zeta2}, \eqref{eq psi L^4}, \eqref{eq psi S^sigma} and Corollary \ref{cor Est zeta1}, we write \eqref{eq zeta2 S^sigma} into
\begin{align*}
\norm{\zeta_2}_{S^{\sigma}(I)} & \lesssim \norm{\zeta_2}_{S^{\sigma}(I)}  s_0^s  + s_0^{\frac{1}{2} (s-\sigma)}  s_0^s  + \norm{\zeta_2}_{S^{\sigma}(I)}  \varepsilon^{\frac{1}{2}} + s_0^{\frac{\sigma}{2} (s-1)} \varepsilon^{\frac{1}{4}} s_0^{\frac{1}{2}s}  + \norm{\zeta_2}_{S^{\sigma}(I)}  \varepsilon^{\frac{1}{2}} + s_0^{\frac{\sigma}{2} (s-1)} \varepsilon^{\frac{1}{4}} s_0^{\frac{1}{2}s} .
\end{align*}
Then we have
\begin{align*}
\norm{\zeta_2}_{S^{\sigma}(I)} \lesssim  s_0^{\frac{\sigma}{2} (s-1)}  s_0^{\frac{1}{2}s} = s_0^{\frac{1}{2} (\sigma s + s -\sigma)}.
\end{align*}

\end{proof}

Finally, we arrive at the proof of {\it (3)} in Lemma \ref{lem Est zeta2}.

\begin{proof}[Proof of {\it (3)} in Lemma \ref{lem Est zeta2}]
In this step, we prove the smoothness of the solution $\zeta_2$ using the local smoothing estimate and the radial assumption of the initial data. In fact, this is the only place where the radial assumption is used, and all other steps work for all general data.

First, by Strichartz inequalities, we write
\begin{align}\label{eq nabla zeta2}
\begin{aligned}
\norm{\nabla \zeta_2}_{L_t^{\infty} L_x^2(I \times \h^2)} & \lesssim \norm{\nabla  \OO( \zeta_2^3) }_{L_{t,x}^{\frac{4}{3}}  (I \times \h^2)}  + \norm{\nabla \OO( \psi^3)}_{L_{t,x}^{\frac{4}{3}}  (I \times \h^2)} \\
& + \norm{\nabla \OO (\zeta_2 \zeta_1^2)}_{L_{t,x}^{\frac{4}{3}}  (I \times \h^2)} + \norm{\nabla  \OO (\psi \zeta_1^2)}_{L_{t,x}^{\frac{4}{3}}  (I \times \h^2)}\\
& : = I + II + III +IV . 
\end{aligned}
\end{align}

Before estimating $I$, $II$, $III$ and $IV$, we compute the following norms that are needed in the rest of the proof:
\begin{claim}\label{claim useful norms} 
For $r = \abs{x}$,
\begin{align}\label{eq useful norms}
& \norm{\sinh^{\frac{1}{2}} (r) \psi}_{L_{t,x}^{\infty} (I \times \h^2)}  \lesssim  s_0^{\frac{1}{2}s -\frac{1}{4}}, \notag\\
& \norm{\sinh^{\frac{1}{2}} (r) \zeta_1}_{L_{t,x}^{\infty} (I \times \h^2)}  \lesssim  s_0^{\frac{1}{4}(s-1)}, \\
& \norm{\sinh^{\frac{1}{2}} (r) \zeta_2}_{L_{t,x}^{\infty} (I \times \h^2)}  \lesssim  s_0^{\frac{1}{2}s -\frac{1}{4}} . \notag
\end{align}
\end{claim}

\begin{proof}[Proof of Claim \ref{claim useful norms}]
Using the radial Sobolev embedding (Corollary \ref{cor H Radial Sobolev}), Strichartz estimates and Fact \ref{fact psi0}
\begin{align*}
\norm{\sinh^{\frac{1}{2}} (r) \psi}_{L_{t,x}^{\infty} (I \times \h^2)} & \lesssim \norm{\psi}_{L_t^{\infty} L_x^2 (I \times \h^2)}^{1-\frac{1}{4\alpha}} \norm{(-\Delta)^{\alpha} \psi }_{L_t^{\infty} L_x^2(I \times \h^2)}^{\frac{1}{4\alpha}}\\
& \lesssim \norm{\psi_0 }_{L_x^2 ( \h^2)}^{1-\frac{1}{4\alpha}} \norm{\psi_0}_{H_x^{2\alpha} ( \h^2)}^{\frac{1}{4\alpha}}\lesssim s_0^{\frac{1}{2}s \times (1-\frac{1}{4\alpha})}  s_0^{\frac{1}{2}(s-2\alpha) \times \frac{1}{4\alpha}} = s_0^{\frac{1}{2}s -\frac{1}{4}} .
\end{align*}

By Corollary \ref{cor H Radial Sobolev} and Corollary \ref{cor Est zeta1}
\begin{align*}
 \norm{\sinh^{\frac{1}{2}} (r) \zeta_1}_{L_{t,x}^{\infty} (I \times \h^2)}  \lesssim \norm{\zeta_1}_{L_t^{\infty} L_x^2 (I \times \h^2)}^{\frac{1}{2}} \norm{\nabla \zeta_1 }_{L_t^{\infty} L_x^2(I \times \h^2)}^{\frac{1}{2}} \lesssim s_0^{\frac{1}{4}(s-1)}  .
\end{align*}

By Corollary \ref{cor H Radial Sobolev} and Lemma \ref{lem zeta2 S^sigma}
\begin{align*}
\norm{\sinh^{\frac{1}{2}} (r) \zeta_2}_{L_{t,x}^{\infty} (I \times \h^2)} & \lesssim \norm{\zeta_2}_{L_t^{\infty} L_x^2 (I \times \h^2)}^{1-\frac{1}{4\alpha}} \norm{(-\Delta)^{\alpha} \zeta_2 }_{L_t^{\infty} L_x^2(I \times \h^2)}^{\frac{1}{4\alpha}} \\
& \lesssim s_0^{\frac{1}{2}s \times (1-\frac{1}{4\alpha})}  s_0^{\frac{1}{2}(2\alpha s+ s-2\alpha) \times \frac{1}{4\alpha}} = s_0^{\frac{3}{4}s -\frac{1}{4}} .
\end{align*}

\end{proof}

Let us also recall some estimates from previous subsections (\eqref{eq psi S^sigma}, \eqref{eq psi L^4}, Corollary \ref{cor Est zeta1}, Lemma \ref{lem zeta2 S^sigma} and Lemma \ref{lem Est zeta2}):
\begin{flalign}\label{eq All estimates}
&\norm{\psi}_{S^{\sigma}(\R)} \lesssim s_0^{\frac{1}{2}(s-\sigma)} \text{ for } 0 \leq \sigma \leq s ,&&  \norm{\psi}_{L_{t,x}^4(\R \times \h^2)}  \lesssim s_0^{\frac{1}{2}s} , & \notag\\
&\norm{\zeta_1}_{S^{\sigma} (I )} \lesssim s_0^{\frac{\sigma}{2}(s-1)} \text{ for } 0 \leq \sigma \leq 1 ,&& \norm{\zeta_1}_{L_{t,x}^4 (I \times \h^2)}^4  \lesssim \varepsilon  , &\\
&\norm{\zeta_2}_{S^{\sigma}(I)} \lesssim s_0^{\frac{1}{2} (s-\sigma)} \text{ for } 0 \leq \sigma \leq s ,&& \norm{\zeta_2}_{L_{t,x}^4 (I \times \h^2)} \lesssim s_0^{\frac{1}{2}s} . & \notag
\end{flalign}

Now we continue working on \eqref{eq nabla zeta2}.  By Lemma \ref{lem zeta2 S^sigma}
\begin{align*}
I = \norm{\nabla  \OO( \zeta_2^3) }_{L_{t,x}^{\frac{4}{3}}  (I \times \h^2)}  \lesssim  \norm{\nabla  \zeta_2 }_{L_t^{\infty} L_x^2 (I \times \h^2)} \norm{  \zeta_2 }_{L_t^{\frac{8}{3}} L_x^8 (I \times \h^2)}^2 \lesssim s_0^s \norm{\nabla  \zeta_2 }_{L_t^{\infty} L_x^2 (I \times \h^2)} .
\end{align*}
This term will be absorbed by the left hand side of \eqref{eq nabla zeta2}.

For $II$, we employ the local smoothing estimate. Since $\psi$ is a linear solution, the linear version should be enough for this term. To implement the local smoothing estimate, we would like to introduce the weight $\inner{x}^{-\frac{1}{2} -\varepsilon_1}$, where $\varepsilon_1$ is a small positive number, and split out half derivative from the full gradient. Then by chain rule and H\"older inequality, we write
\begin{align}\label{eq psi^3}
\norm{\nabla \OO(\psi^3)}_{L_{t,x}^{\frac{4}{3}} (I \times \h^2)} & \lesssim  \norm{\inner{x}^{-\frac{1}{2} -\varepsilon_1} \abs{\nabla}^{\frac{1}{2}} (\abs{\nabla}^{\frac{1}{2}} \psi)}_{L_{t,x}^2(I_\times \h^2)} \norm{\inner{x}^{\frac{1}{2}+\varepsilon_1} \OO(\psi^2)}_{L_{t,x}^4 (I \times \h^2)} 
\end{align}
Now we compute the two factors above separately. Using the linear local smoothing estimate (Lemma \ref{lem H LSE}) and Lemma \ref{lem Bernstein}, we write the first factor into
\begin{align}\label{eq psi^3 1}
\norm{\inner{x}^{-\frac{1}{2} -\varepsilon_1} \abs{\nabla}^{\frac{1}{2}} (\abs{\nabla}^{\frac{1}{2}} \psi)}_{L_{t,x}^2(I \times \h^2)} \lesssim  \norm{\abs{\nabla}^{\frac{1}{2}} P_{\leq s_0} \phi}_{L_x^2(\h^2)} \lesssim s_0^{\frac{1}{2} (s-\frac{1}{2})}.
\end{align}

To estimate the second factor in \eqref{eq psi^3}, by H\"older inequality, Sobelev embedding, \eqref{eq psi S^sigma}, \eqref{eq psi L^4} and Claim \ref{claim useful norms}, we have
\begin{align}
\norm{\inner{x}^{\frac{1}{2}+\varepsilon_1} \OO(\psi^2)}_{L_{t,x}^4 (I \times \h^2)} & \lesssim \norm{\chi_{\{ \abs{x} \leq 1 \}} \OO(\psi^2)}_{L_{t,x}^4 (I \times \h^2)} + \norm{\chi_{\{ \abs{x} >1 \}}  \abs{x}^{\frac{1}{2}+\varepsilon_1}  \OO(\psi^2) }_{L_{t,x}^4 (I \times \h^2)}  \notag\\
& \lesssim \norm{\psi}_{L_{t,x}^8 (I \times \h^2)}^2 + \norm{\sinh^{\frac{1}{2}} (r) \psi}_{L_{t,x}^{\infty} (I \times \h^2)} \norm{ \psi}_{L_{t,x}^4 (I \times \h^2)} \notag\\
& \lesssim s_0^{s-\frac{1}{2}} + s_0^{\frac{1}{2}s -\frac{1}{4}} s_0^{\frac{1}{2}s} \lesssim s_0^{s-\frac{1}{2}} , \label{eq psi^3 2}
\end{align}
for all $s$.
Then combining \eqref{eq psi^3 1} and \eqref{eq psi^3 2}, we obtain
\begin{align*}
II \lesssim \eqref{eq psi^3} \lesssim  s_0^{\frac{1}{2} (s-\frac{1}{2})} s_0^{s-\frac{1}{2}} = s_0^{\frac{3}{2}s - \frac{3}{4}} .
\end{align*}

For $III$ and $IV$, we write them in a similar way as in \eqref{eq psi^3}. 
\begin{align}\label{eq term III}
III & \lesssim \norm{\nabla \OO (\zeta_2 \zeta_1^2)}_{L_{t,x}^{\frac{4}{3}}  (I \times \h^2)} \notag\\
& \lesssim  \norm{\inner{x}^{-\frac{1}{2} -\varepsilon_1} \abs{\nabla}^{\frac{1}{2}} (\abs{\nabla}^{\frac{1}{2}} \zeta_2)}_{L_{t,x}^2(I \times \h^2)} \norm{\inner{x}^{\frac{1}{2}+\varepsilon_1} \OO(\zeta_1^2)}_{L_{t,x}^4 (I \times \h^2)} \\
& \quad + \norm{\inner{x}^{-\frac{1}{2} -\varepsilon_1} \abs{\nabla}^{\frac{1}{2}} (\abs{\nabla}^{\frac{1}{2}} \zeta_1)}_{L_{t,x}^2(I \times \h^2)} \norm{\inner{x}^{\frac{1}{2}+\varepsilon_1} \OO(\zeta_1 \zeta_2)}_{L_{t,x}^4 (I \times \h^2)} ,\notag
\end{align}
\begin{align}\label{eq term IV}
IV & \lesssim \norm{\nabla \OO (\psi \zeta_1^2)}_{L_{t,x}^{\frac{4}{3}}  (I \times \h^2)} \notag\\
& \lesssim  \norm{\inner{x}^{-\frac{1}{2} -\varepsilon_1} \abs{\nabla}^{\frac{1}{2}} (\abs{\nabla}^{\frac{1}{2}} \psi)}_{L_{t,x}^2(I \times \h^2)} \norm{\inner{x}^{\frac{1}{2}+\varepsilon_1} \OO(\zeta_1^2)}_{L_{t,x}^4 (I \times \h^2)} \\
& \quad + \norm{\inner{x}^{-\frac{1}{2} -\varepsilon_1} \abs{\nabla}^{\frac{1}{2}} (\abs{\nabla}^{\frac{1}{2}} \zeta_1)}_{L_{t,x}^2(I \times \h^2)} \norm{\inner{x}^{\frac{1}{2}+\varepsilon_1} \OO(\zeta_1 \psi)}_{L_{t,x}^4 (I \times \h^2)} .\notag
\end{align}

Due to the mixed terms in $III$ and $IV$, the calculation will be more technical. Noting that there are some common terms in $III$ and $IV$, we will preform the estimation of $III$ and $IV$ at the same time. Our goal here is to prove that $III$ is bounded by $s_0^{s-\frac{3}{4}}$ and $IV$ is also by $s_0^{s-\frac{3}{4}}$.

We first start with the terms with no derivative and positive weights.
By H\"older inequality, Sobolev embedding, Corollary \ref{cor Est zeta1} and \eqref{eq useful norms}
\begin{align*}
\norm{\inner{x}^{\frac{1}{2}+\varepsilon_1} \OO(\zeta_1^2)}_{L_{t,x}^4 (I \times \h^2)} \lesssim  \norm{\zeta_1}_{L_{t,x}^8 (I \times \h^2)}^2 + \norm{\sinh^{\frac{1}{2}} (r) \zeta_1}_{L_{t,x}^{\infty} (I \times \h^2)} \norm{\zeta_1}_{L_{t,x}^4 (I \times \h^2)} \\
\lesssim \norm{\inner{-\Delta}^{\frac{1}{4}} \zeta_1}_{L_{t}^8 L_x^{\frac{8}{3}}(I \times \h^2)}^2 + s_0^{\frac{1}{4} (s-1)} \varepsilon^{\frac{1}{4}} \lesssim s_0^{\frac{1}{2}(s-1)}+ s_0^{\frac{1}{4} (s-1)} \lesssim s_0^{\frac{1}{2}(s-1)} ,
\end{align*}
for $s < 1$.

By H\"older inequality, Sobolev embedding, Corollary \ref{cor Est zeta1}, Lemma \ref{lem Est zeta2}, Lemma \ref{lem zeta2 S^sigma} and \eqref{eq useful norms}
\begin{align*}
\norm{\inner{x}^{\frac{1}{2}+\varepsilon_1} \OO(\zeta_1 \zeta_2)}_{L_{t,x}^4 (I \times \h^2)}  & \lesssim \norm{\zeta_1}_{L_{t,x}^{8} (I \times \h^2)} \norm{\zeta_2}_{L_{t,x}^8 (I \times \h^2)} + \norm{\sinh^{\frac{1}{2}} (r) \zeta_2}_{L_{t,x}^{\infty} (I \times \h^2)} \norm{\zeta_1}_{L_{t,x}^4 (I \times \h^2)}\\
& \lesssim \norm{\inner{-\Delta}^{\frac{1}{4}}\zeta_1}_{L_{t}^8 L_x^{\frac{8}{3}} (I \times \h^2)} \norm{\inner{-\Delta}^{\frac{1}{4}} \zeta_2}_{L_{t}^8 L_x^{\frac{8}{3}} (I \times \h^2)} + s_0^{\frac{3}{4}s-\frac{1}{4}} \varepsilon^{\frac{1}{4}}\\
& \lesssim  s_0^{\frac{1}{4} (s-1)} s_0^{\frac{1}{2}(\frac{3}{2} s-\frac{1}{2})} + s_0^{\frac{3}{4}s-\frac{1}{4}} \varepsilon^{\frac{1}{4}} \lesssim s_0^{s-\frac{1}{2}} ,
\end{align*}
for $s < 1$.

By H\"older inequality, Sobolev embedding, Corollary \ref{cor Est zeta1}, \eqref{eq psi S^sigma} and \eqref{eq useful norms}
\begin{align*}
\norm{\inner{x}^{\frac{1}{2}+\varepsilon_1} \OO(\zeta_1 \psi)}_{L_{t,x}^4 (I \times \h^2)}  & \lesssim \norm{\zeta_1}_{L_{t,x}^{12} (I \times \h^2)} \norm{\psi}_{L_{t,x}^6 (I \times \h^2)} + \norm{\sinh^{\frac{1}{2}} (r) \psi}_{L_{t,x}^{\infty} (I \times \h^2)} \norm{\zeta_1}_{L_{t,x}^4 (I \times \h^2)}\\
& \lesssim \norm{\inner{-\Delta}^{\frac{1}{3}}\zeta_1}_{L_{t}^{12} L_x^{\frac{12}{5}} (I \times \h^2)} \norm{\inner{-\Delta}^{\frac{1}{6}} \psi}_{L_{t}^6 L_x^{3} (I \times \h^2)} + s_0^{\frac{1}{2}s-\frac{1}{4}} \varepsilon^{\frac{1}{4}}\\
& \lesssim  s_0^{\frac{1}{3} (s-1)} s_0^{\frac{1}{2}( s-\frac{1}{3})} + s_0^{\frac{1}{2}s-\frac{1}{4}}  \varepsilon^{\frac{1}{4}} \lesssim s_0^{\frac{1}{2} s-\frac{1}{4}} ,
\end{align*}
for $s > \frac{3}{4}$.

Then we focus on the terms with derivatives and negative weights. 

Notice that we have treated $\norm{\inner{x}^{-\frac{1}{2} -\varepsilon_1} \abs{\nabla}^{\frac{1}{2}} (\abs{\nabla}^{\frac{1}{2}} \psi)}_{L_{t,x}^2(I \times \h^2)} $ in \eqref{eq psi^3 1}.

By H\"older inequality, Fact \ref{fact eta0}, Sobolev embedding, Corollary \ref{cor Est zeta1} and \eqref{eq useful norms}
\begin{align*}
\norm{\inner{x}^{-\frac{1}{2} -\varepsilon_1} \abs{\nabla}^{\frac{1}{2}} (\abs{\nabla}^{\frac{1}{2}} \zeta_1)}_{L_{t,x}^2(I \times \h^2)} & \lesssim \norm{\abs{\nabla}^{\frac{1}{2}} \eta_0}_{L_x^2(\h^2)} + \norm{\inner{x}^{\frac{1}{2} +\varepsilon_1} \abs{\zeta_1}^2 \zeta_1}_{L_{t,x}^2 (I \times \h^2)} \\
& \lesssim s_0^{\frac{1}{4}(s-1)} + \norm{\zeta_1 }_{L_{t,x}^6 (I \times \h^2)}^3 + \norm{\sinh^{\frac{1}{2}} (r) \zeta_1}_{L_{t,x}^{\infty} (I \times \h^2)} \norm{\zeta_1}_{L_{t,x}^4 (I \times \h^2)}^2\\
& \lesssim s_0^{\frac{1}{4}(s-1)}  +\norm{ \inner{-\Delta}^{\frac{1}{6}} \zeta_1 }_{L_{t}^6 L_x^3 (I \times \h^2)}^3 + s_0^{\frac{1}{4}(s-1)} \varepsilon^{\frac{1}{2}} \\
& \lesssim s_0^{\frac{1}{4}(s-1)}  +s_0^{\frac{1}{6}(s-1) \times 3} + s_0^{\frac{1}{4}(s-1)} \varepsilon^{\frac{1}{2}} \lesssim s_0^{\frac{1}{2}(s-1)} ,
\end{align*}
for $s < 1$.

The last term in this category needs more work. By Corollary \ref{cor H Radial Sobolev} and triangle inequality, we have
\begin{align}\label{eq hard term}
\norm{\inner{x}^{-\frac{1}{2} -\varepsilon_1} \abs{\nabla}^{\frac{1}{2}} (\abs{\nabla}^{\frac{1}{2}} \zeta_2)}_{L_{t,x}^2(I \times \h^2)} & \lesssim \norm{\inner{x}^{\frac{1}{2}+\varepsilon_1}  F}_{L_{t,x}^2 (I \times \h^2)} \\
& \lesssim \norm{\inner{x}^{\frac{1}{2}+\varepsilon_1}   \OO( \zeta_2^3) }_{L_{t,x}^2 (I \times \h^2)}  + \norm{\inner{x}^{\frac{1}{2}+\varepsilon_1}  \OO( \psi^3)}_{L_{t,x}^2 (I \times \h^2)} \notag\\
& + \norm{\inner{x}^{\frac{1}{2}+\varepsilon_1}  \OO (\zeta_2 \zeta_1^2)}_{L_{t,x}^2 (I \times \h^2)} + \norm{\inner{x}^{\frac{1}{2}+\varepsilon_1}  \OO (\psi \zeta_1^2)}_{L_{t,x}^2 (I \times \h^2)}. \notag
\end{align}
Next we will estimate these four terms above separately. By H\"older inequality, Sobolev embedding, Corollary \ref{cor Est zeta1}, Lemma \ref{lem Est zeta2}, Lemma \ref{lem zeta2 S^sigma}, \eqref{eq psi S^sigma}, \eqref{eq psi L^4}  and \eqref{eq useful norms}
\begin{align*}
\norm{\inner{x}^{\frac{1}{2}+\varepsilon_1}   \OO( \zeta_2^3) }_{L_{t,x}^2 (I \times \h^2)} & \lesssim \norm{\zeta_2}_{L_{t,x}^6 (I \times \h^2)}^3 +  \norm{\sinh^{\frac{1}{2}} (r) \zeta_2}_{L_{t,x}^{\infty} (I \times \h^2)} \norm{\zeta_2}_{L_{t,x}^4 (I \times \h^2)}^2 \\
& \lesssim \norm{ \inner{-\Delta}^{\frac{1}{6}} \zeta_2 }_{L_{t}^6 L_x^3 (I \times \h^2)}^3+ s_0^{\frac{3}{4}s -\frac{1}{4}} s_0^s \lesssim s_0^{\frac{1}{2} (\frac{4}{3}s -\frac{1}{3}) \times 3}  + s_0^{\frac{3}{4}s -\frac{1}{4}} s_0^s \lesssim s_0^{2s -\frac{1}{2}} ,
\end{align*}
for $s<1$.

\begin{align*}
\norm{\inner{x}^{\frac{1}{2}+\varepsilon_1}  \OO( \psi^3)}_{L_{t,x}^2 (I \times \h^2)} & \lesssim \norm{\psi}_{L_{t,x}^6 (I \times \h^2)}^3 +  \norm{\sinh^{\frac{1}{2}} (r) \psi}_{L_{t,x}^{\infty} (I \times \h^2)} \norm{\psi}_{L_{t,x}^4 (I \times \h^2)}^2 \\
& \lesssim \norm{ \inner{-\Delta}^{\frac{1}{6}} \psi }_{L_{t}^6 L_x^3 (I \times \h^2)}^3 + s_0^{\frac{1}{2}s -\frac{1}{4}} s_0^s \lesssim s_0^{\frac{1}{2}(s-\frac{1}{3}) \times 3} + s_0^{\frac{1}{2}s -\frac{1}{4}} s_0^s  \lesssim s_0^{\frac{3}{2}s -\frac{1}{2}} ,
\end{align*}
for all $s $.

\begin{align*}
\norm{\inner{x}^{\frac{1}{2}+\varepsilon_1}  \OO (\zeta_2 \zeta_1^2)}_{L_{t,x}^2 (I \times \h^2)} & \lesssim \norm{\zeta_2}_{L_{t,x}^6 (I \times \h^2)} \norm{\zeta_1}_{L_{t,x}^6 (I \times \h^2)}^2  +  \norm{\sinh^{\frac{1}{2}} (r) \zeta_2}_{L_{t,x}^{\infty} (I \times \h^2)} \norm{\zeta_1}_{L_{t,x}^4 (I \times \h^2)}^2 \\
& \lesssim \norm{ \inner{-\Delta}^{\frac{1}{6}} \zeta_2 }_{L_{t}^6 L_x^3 (I \times \h^2)} \norm{ \inner{-\Delta}^{\frac{1}{6}} \zeta_1 }_{L_{t}^6 L_x^3 (I \times \h^2)}^2 + s_0^{\frac{3}{4}s - \frac{1}{4}} \varepsilon^{\frac{1}{2}}\\
& \lesssim  s_0^{\frac{1}{2}(\frac{4}{3}s -\frac{1}{3})} s_0^{\frac{1}{6}(s-1) \times 2}  + s_0^{\frac{3}{4}s - \frac{1}{4}} \varepsilon^{\frac{1}{2}} \lesssim s_0^{s-\frac{1}{2}},
\end{align*}
for $s<1$.

\begin{align*}
\norm{\inner{x}^{\frac{1}{2}+\varepsilon_1}  \OO (\psi \zeta_1^2)}_{L_{t,x}^2 (I \times \h^2)} & \lesssim \norm{\psi}_{L_{t,x}^6 (I \times \h^2)} \norm{\zeta_1}_{L_{t,x}^6 (I \times \h^2)}^2 +  \norm{\sinh^{\frac{1}{2}} (r) \psi}_{L_{t,x}^{\infty} (I \times \h^2)} \norm{\zeta_1}_{L_{t,x}^4 (I \times \h^2)}^2 \\
& \lesssim \norm{ \inner{-\Delta}^{\frac{1}{6}} \psi }_{L_{t}^6 L_x^3 (I \times \h^2)} \norm{ \inner{-\Delta}^{\frac{1}{6}} \zeta_1}_{L_{t}^6 L_x^3 (I \times \h^2)}^2 + s_0^{\frac{1}{2}s - \frac{1}{4}} \varepsilon^{\frac{1}{2}} \\
& \lesssim s_0^{\frac{1}{2}(s -\frac{1}{3})} s_0^{\frac{1}{6}(s-1) \times 2}  + s_0^{\frac{1}{2}s - \frac{1}{4}} \varepsilon^{\frac{1}{2}} \lesssim s_0^{\frac{1}{2}s -\frac{1}{4}},
\end{align*}
for $s > \frac{3}{4}$.

Then summing up all these four terms gives
\begin{align*}
\eqref{eq hard term} \lesssim s_0^{2s -\frac{1}{2}} + s_0^{\frac{3}{2}s -\frac{1}{2}} + s_0^{s-\frac{1}{2}} +  s_0^{\frac{1}{2}s -\frac{1}{4}} \lesssim s_0^{\frac{1}{2} s - \frac{1}{4}} ,
\end{align*}
for $s> \frac{1}{4}$.

Therefore, continue from \eqref{eq term III} and \eqref{eq term IV}
\begin{align*}
III & \lesssim s_0^{\frac{1}{2} s - \frac{1}{4}} s_0^{\frac{1}{2}(s-1)} + s_0^{\frac{1}{2}(s-1)} s_0^{s-\frac{1}{2}} \lesssim s_0^{s-\frac{3}{4}} ,  \quad \text{ for } s > \frac{1}{4},\\
IV & \lesssim s_0^{\frac{1}{2} (s-\frac{1}{2})} s_0^{\frac{1}{2}(s-1)} + s_0^{\frac{1}{2}(s-1)} s_0^{\frac{1}{2} s -\frac{1}{4}} \eqsim s_0^{s-\frac{3}{4}} , \quad \text{ for all } s.
\end{align*}

Combining all the terms $I$, $II$, $III$ and $IV$, we write \eqref{eq nabla zeta2} into
\begin{align*}
\norm{\nabla \zeta_2}_{L_t^{\infty} L_x^2(I \times \h^2)}   \lesssim I + II + III + IV \lesssim s_0^s \norm{\nabla  \zeta_2 }_{L_t^{\infty} L_x^2 (I \times \h^2)} +  s_0^{\frac{3}{2}s - \frac{3}{4}}  + s_0^{s-\frac{3}{4}} + s_0^{s-\frac{3}{4}} .
\end{align*}
Then we have
\begin{align*}
\norm{\nabla \zeta_2}_{L_t^{\infty} L_x^2(I \times \h^2)}   \lesssim  s_0^{s-\frac{3}{4}} .
\end{align*}

This concludes  the proof of {\it (3)} in Lemma \ref{lem Est zeta2} and complete the analysis  of the energy increment.
\end{proof}

\subsection{Proof of Proposition \ref{prop Global}}
We divide the time interval $[0,T]$ into $[0,T] = \cup_i I_i = \cup_i [a_i , a_{i+1}] $, such that on each $I_i$, $\norm{u}_{L_{t,x}^4 (I_i \times \h^2)}^4 = \varepsilon$. Hence
\begin{align*}
\# I_i \sim \frac{M}{\varepsilon}.
\end{align*}
Let us remark that the length of such small intervals could be very long, and if some of them is an infinite interval, say $[a_k, \infty)$, then we just call $a_{k+1} = \infty$. 

On the first interval $I_1= [0, a_1]$, we can apply Proposition \ref{prop Local} and have the local energy increment
\begin{align*}
E(\zeta (a_1) ) \leq E(\zeta(0)) + C s_0^{\frac{3}{2}s -\frac{5}{4}}  .
\end{align*}
On the second time interval $I_2= [a_1, a_2]$, we solve $\zeta$ by solving  a cubic NLS with smoother data,
\begin{align}\label{eq zeta11}
\begin{cases}
i \partial_t \zeta_1^{(1)} + \Delta_{\h^2} \zeta_1^{(1)} = \abs{\zeta_1^{(1)}}^2 \zeta_1^{(1)} , \\
\zeta_1^{(1)}(a_1,x) = \zeta(a_1) = \zeta_1 (a_1) + \zeta_2 (a_1),
\end{cases}
\end{align}
and a difference equation with zero initial value,
\begin{align}\label{eq zeta22}
\begin{cases}
i \partial_t \zeta_2^{(1)} + \Delta_{\h^2} \zeta_2^{(1)} = \abs{u}^2 u - \abs{\zeta_1^{(1)}}^2 \zeta_1^{(1)} , \\
\zeta_2^{(1)}(a_1,x) =0.
\end{cases}
\end{align}
Hence $\zeta = \zeta_1^{(1)} + \zeta_2^{(1)}$ and the full solution will be $u= \zeta +  \psi= \zeta_1^{(1)} + \zeta_2^{(1)} + \psi $. Applying Proposition \ref{prop Local} again, we see that
\begin{align*}
E(\zeta (a_2) ) \leq E(\zeta(a_1)) + C s_0^{\frac{3}{2}s -\frac{5}{4}}  .
\end{align*}
The reason why we are safe to apply Proposition \ref{prop Local} on $I_2$ is that new decomposed initial data $\zeta(a_1) $ and $\psi(a_1)$ satisfy all the facts in Facts \ref{fact eta0} and \ref{fact psi0}, and all the calculations that we did in Proposition \ref{prop Local} will apply to the new systems \eqref{eq zeta11} and \eqref{eq zeta22}. In particular, the size of new initial data $\zeta(a_1) $ in energy is the size of $\zeta(0) $ in \eqref{eq zeta} plus a small error from $\zeta_2 (a_1)$, which can be seen from Proposition \ref{prop Local}
\begin{align*}
E(\zeta (a_1) ) & \leq E(\zeta(0)) + \underbrace{C s_0^{\frac{3}{2}s -\frac{5}{4}}}_{\text{small error}}  \sim E(\zeta(0)) \lesssim s_0^{s-1}.
\end{align*}
Also the $H^1$ norm of $\zeta(a_1) $ can be thought as the $H^1$ norm of $\zeta(0) $ plus a small error,
\begin{align*}
\norm{\zeta(a_1) }_{H_x^1 (\h^2)} & \leq  \norm{\zeta_1 (a_1)}_{H_x^1 (\h^2)} + \norm{ \zeta_2 (a_1) }_{H_x^1  (\h^2)} \lesssim \underbrace{s_0^{\frac{1}{2}(s-1)}}_{\text{size of $H^1$ norm of $\zeta(0) $}} + \underbrace{s_0^{s-\frac{3}{4}} }_{\text{small error}} \sim s_0^{\frac{1}{2}(s-1)}.
\end{align*}

Then we can continue this iteration as long as the accumulated energy increment does not surpass the size of the initial energy of $\zeta (0)$, which guarantees that the setup for the smoother component remains the same size in the next iteration. That is,
\begin{align*}
C \frac{M}{\varepsilon} s_0^{\frac{3}{2}s -\frac{5}{4}} \leq E(\zeta(0)) \sim s_0^{s-1},
\end{align*}
which gives
\begin{align}\label{eq choice}
M \sim s_0^{-\frac{1}{2}s + \frac{1}{4}} .
\end{align}
And the total energy increment is
\begin{align}\label{eq Energy increment}
E(\zeta (T) ) \leq E(\zeta(0)) + C \frac{M}{\varepsilon}  s_0^{\frac{3}{2}s -\frac{5}{4}}  .
\end{align}
For now, we finish the proof of Proposition \ref{prop Global} and give the choice of $s_0$.

\begin{rmk}[boundedness of $H^s$ norm of $u$]\label{rmk u H^s}
As a consequence of Proposition \ref{prop Global}, we conclude that the $H^s$ norm of $u$ has the following bound
\begin{align*}
\norm{u(T) }_{H_x^s (\h^2)}   \lesssim s_0^{\frac{1}{2} (s-1) s}  .
\end{align*}
In fact, \eqref{eq Energy increment} implies the boundedness of the $H^1$ norm of $\zeta(T)$,
\begin{align}\label{eq zeta H^1}
\norm{\zeta(T) }_{H_x^1 (\h^2)}^2  \leq E(\zeta (T) ) \leq E(\zeta(0)) + C \frac{M}{\varepsilon}  s_0^{\frac{3}{2}s -\frac{5}{4}}  \lesssim s_0^{s-1} .
\end{align}
And triangle inequality and the mass conservation laws of $u$ and $\psi$ with \eqref{eq psi S^sigma} give the boundedness of $L^2$ norm of $\zeta(T)$
\begin{align}\label{eq zeta L^2}
\norm{\zeta(T) }_{L_x^2 (\h^2)} & \leq \norm{u(T) }_{L_x^2 (\h^2)} + \norm{\psi(T) }_{L_x^2 (\h^2)} =  \norm{u(0) }_{L_x^2 (\h^2)} +  \norm{\psi(0) }_{L_x^2 (\h^2)} \notag\\
& \lesssim  \norm{u(0) }_{L_x^2 (\h^2)} + s_0^{\frac{1}{2}s} \leq 2 \norm{u(0) }_{L_x^2 (\h^2)} .
\end{align}
Then the $H^s$  bound $\zeta(T)$ follows from the interpolation \eqref{eq zeta L^2} and \eqref{eq zeta H^1}
\begin{align}\label{eq zeta H^s}
\norm{\zeta(T) }_{H_x^s (\h^2)} \lesssim \norm{\zeta(T) }_{L_x^2 (\h^2)}^{1-s} \norm{\zeta(T) }_{H_x^1 (\h^2)}^{s} \lesssim s_0^{\frac{1}{2} (s-1) s}.
\end{align}
Therefore the $H^s$ norm of  $u(T)$ is bounded due \eqref{eq zeta H^s} and the fact $\psi \in H^s (\h^2)$,
\begin{align*}
\norm{u(T) }_{H_x^s (\h^2)}  \leq \norm{\zeta(T) }_{H_x^s (\h^2)} + \norm{\psi(T) }_{H_x^s (\h^2)} \lesssim s_0^{\frac{1}{2} (s-1) s} + 1  \lesssim s_0^{\frac{1}{2} (s-1) s}  .
\end{align*}
Consequently, we also have the bound of the $H^{\sigma} (0 < \sigma < s)$ norm by interpolating the $H^s$ with $L^2$ norms
\begin{align*}
\norm{u(T) }_{H_x^{\sigma} (\h^2)} \lesssim \norm{u(T) }_{H_x^s (\h^2)}^{\frac{\sigma}{s}} \norm{u(T) }_{L_x^2 (\h^2)}^{1- \frac{\sigma}{2}} \lesssim s_0^{\frac{1}{2} (s-1) \sigma}.
\end{align*}
\end{rmk}

\section{Morawetz estimates on $\h^2$}\label{sec Morawetz}
Recall that the Morawetz estimate of the cubic NLS on $\h^2$ in \cite{IS}, when $u$ is the solution to the cubic NLS equation  $i \partial_t u + \Delta_{\h^2} u = \abs{u}^2 u$, reads as
\begin{align*}
\norm{u}_{L_{t,x}^4([t_1,t_2] \times \h^2)}^4  \lesssim \norm{u}_{L_t^{\infty} L_x^2 ([t_1,t_2] \times \h^2)} \norm{u}_{L_t^{\infty} H_x^1 ([t_1,t_2]\times \h^2)} .
\end{align*}

\begin{prop}\label{prop Morawetz}
If we modify the NLS equation, that is $u$ solves
\begin{align*}
i \partial_t u + \Delta_{\h^2} u = \abs{u}^2 u + \NN ,
\end{align*}
then the modified Morawetz estimate becomes
\begin{align}
\begin{aligned}\label{modmor}
\norm{u}_{L_{t,x}^4([t_1,t_2] \times \h^2)}^4 & \lesssim \norm{u}_{L_t^{\infty} L_x^2 ([t_1,t_2] \times \h^2)} \norm{u}_{L_t^{\infty} H_x^1 ([t_1,t_2] \times \h^2)} \\
& \quad + \norm{\NN\bar{u}}_{L_{t,x}^1 ([t_1,t_2] \times \h^2)} + \norm{\NN \nabla \bar{ u}}_{L_{t,x}^1 ([t_1,t_2] \times \h^2)} .
\end{aligned}
\end{align}
\end{prop}

\begin{rmk}
The proof of Proposition \ref{prop Morawetz} is very similar as the proof in \cite{IS}. We report it below for the convenience of the reader. The difference is that we consider a more general nonlinear term, which mainly gives two extra terms that account for the two extra terms in \eqref{modmor}.
\end{rmk}

\begin{proof}[Proof of Proposition \ref{prop Morawetz}]
It is possible to construct a function $a(x)$ satisfying the following requirements.
\begin{lem}\label{lem H a}\cite{IS}
There is a smooth function $a : \h^2 \to [0,\infty)$ with the following properties:
\begin{align*}
\begin{aligned}
& \Delta a =1 \text{ on } \h^2\\
& \abs{\nabla a} = \abs{\D^{\alpha} a \D_{\alpha} a}^{\frac{1}{2}} \leq C \text{ on } \h^2\\
& \D^2 a \geq 0 \text{ on } \h^2.
\end{aligned}
\end{align*}
\end{lem}
For $\varepsilon \in (0, \frac{1}{10}]$ let $u_{\varepsilon} = P_{\varepsilon} u$, where $P_{\varepsilon}$ is the smoothing operator defined by the Fourier multiplier $\lambda \to e^{-\varepsilon^2 \lambda^2}$. We fix a smooth even function $\eta_0 : \R \to [0,1]$ supported in $[-2,-\frac{1}{2}] \cup [\frac{1}{2} , 2]$ with the property that
\begin{align*}
\sum_{j \in \Z} \eta_0 (\frac{\lambda}{2^j}) =1 \text{ for any } \lambda \in \R \setminus \{ 0 \}.
\end{align*}
For any $j \in \Z$, let $\eta_j (\lambda) = \eta_0 (\frac{\lambda}{2^j})$ and $\eta_{\leq j} = \sum_{j' \leq j} \eta_{j'}$. Let $\psi_{\varepsilon} : \h^2 \to [0,1]$, $\psi_{\varepsilon}(x) = \eta_{\leq 0} (\varepsilon r)$. 

With $a$ as in Lemma \ref{lem H a}, we define the Morawetz action $M_a : \R \to \R$,
\begin{align*}
M_a (t) = 2 \im \int_{\h^2} \psi_{\varepsilon} \D^{\alpha} a(x) \cdot \bar{u}_{\varepsilon}(x) \D_{\alpha} u_{\varepsilon} (x) \, d\mu(x) .
\end{align*}
Let
\begin{align*}
f_{\varepsilon} = P_{\varepsilon} (\abs{u}^2 u + \NN), 
\end{align*}
thus
\begin{align*}
\partial_t u_{\varepsilon} = i \Delta u_{\varepsilon} - i f_{\varepsilon} \text{ and } \partial_t \bar{u}_{\varepsilon} = -i \Delta \bar{u}_{\varepsilon} + i \bar{f}_{\varepsilon} .
\end{align*}

We compute
\begin{align*}
\partial_t M_a (t) & = 2 \im \int_{\h^2} \psi_{\varepsilon} \D^{\alpha} a \cdot (\partial_t \bar{u}_\varepsilon \cdot \D_{\alpha} u_{\varepsilon} + \bar{u}_{\varepsilon} \cdot \D_{\alpha} \partial_t u_{\varepsilon}) \, d \mu \\
& = 2 \im \int_{\h^2} \psi_{\varepsilon} \D^{\alpha} a \cdot ((-i \Delta \bar{u}_{\varepsilon} + i \bar{f}_{\varepsilon}) \cdot \D_{\alpha} u_{\varepsilon} + \bar{u}_{\varepsilon} \cdot \D_{\alpha}  (i \Delta u_{\varepsilon} - i f_{\varepsilon} )) \, d \mu \\
& = 2 \re \int_{\h^2} \psi_{\varepsilon} \D^{\alpha} a \cdot ((- \Delta \bar{u}_{\varepsilon} + \bar{f}_{\varepsilon}) \cdot \D_{\alpha} u_{\varepsilon} + \bar{u}_{\varepsilon} \cdot \D_{\alpha}  (\Delta u_{\varepsilon} -  f_{\varepsilon} )) \, d \mu \\
& = \int_{\h^2} \psi_{\varepsilon} \D^{\alpha} a \cdot (\bar{u}_{\varepsilon} \cdot \D_{\alpha} \Delta u_{\varepsilon} + u_{\varepsilon} \cdot \D_{\alpha} \Delta \bar{u}_{\varepsilon} - \Delta \bar{u}_{\varepsilon} \cdot \D_{\alpha} u_{\varepsilon} - \Delta u_{\varepsilon} \cdot \D_{\alpha} \bar{u}_{\varepsilon}  ) \, d \mu\\
& \quad + \int_{\h^2} \psi_{\varepsilon} \D^{\alpha} a \cdot (\bar{f}_{\varepsilon} \cdot \D_{\alpha} u_{\varepsilon} + f_{\varepsilon} \cdot \D_{\alpha} \bar{u}_{\varepsilon} - \bar{u}_{\varepsilon} \cdot \D_{\alpha} f_{\varepsilon} - u_{\varepsilon} \cdot \D_{\alpha} \bar{f}_{\varepsilon} ) \, d \mu \\
& = I + II
\end{align*}

By integration by parts and using $\D_{\alpha} \D_{\beta} v =  \D_{\beta} \D_{\alpha} v$ for any scalar $v$, we compute
\begin{align*}
I & = \int_{\h^2} - [\D_{\alpha} (\psi_{\varepsilon} \D^{\alpha} a] (\bar{u}_{\varepsilon}\Delta u_{\varepsilon} + u_{\varepsilon} \Delta \bar{u}_{\varepsilon}) - 2 \psi_{\varepsilon} \D^{\alpha} a ( \Delta \bar{u}_{\varepsilon} \cdot \D_{\alpha} u_{\varepsilon} + \Delta u_{\varepsilon} \cdot \D_{\alpha} \bar{u}_{\varepsilon}) \, d \mu \\
& = \int_{\h^2} - (\psi_{\varepsilon} \Delta a + \D_{\alpha} \psi_{\varepsilon} \D^{\alpha} a) [\Delta (u_{\varepsilon} \bar{u}_{\varepsilon} ) - 2\D_{\beta} u_{\varepsilon} \D^{\beta} \bar{u}_{\varepsilon} ] \, d \mu\\
& \quad - 2 \int_{\h^2} \psi_{\varepsilon} \D^{\alpha} a (\D^{\beta} \D_{\beta} \bar{u}_{\varepsilon} \cdot \D_{\alpha} u_{\varepsilon} + \D^{\beta} \D_{\beta} u_{\varepsilon} \cdot \D_{\alpha} \bar{u}_{\varepsilon} ) \, d \mu\\
& = 2 \int_{\h^2} \D^{\beta} (\psi_{\varepsilon} \D^{\alpha} a) \cdot (\D_{\beta} \bar{u}_{\varepsilon}  \D_{\alpha} u_{\varepsilon} + \D_{\beta} u_{\varepsilon} \D_{\alpha} \bar{u}_{\varepsilon} ) \\
& \quad + \int_{\h^2} - \Delta (\psi_{\varepsilon} \Delta a + \D_{\alpha} \psi_{\varepsilon}  \D^{\alpha} a) \cdot (u_{\varepsilon} \bar{u}_{\varepsilon} ) \, d \mu\\
& \quad + \int_{\h^2} 2 (\psi_{\varepsilon} \Delta a + \D_{\alpha} \psi_{\varepsilon}  \D^{\alpha} a) \cdot \D_{\beta} u_{\varepsilon} \D^{\beta} \bar{u}_{\varepsilon} + 2 \psi_{\varepsilon} \D^{\alpha} a (\D_{\beta} \bar{u}_{\varepsilon} \cdot \D_{\alpha} \D^{\beta} u_{\varepsilon} + \D_{\beta} u_{\varepsilon} \cdot \D_{\alpha} \D^{\beta} \bar{u}_{\varepsilon}) \, d \mu\\
& = 2 \int_{\h^2} (\psi_{\varepsilon} \D^{\beta} \D^{\alpha} a + \D^{\beta} \psi_{\varepsilon} \D^{\alpha} a) (\D_{\beta} \bar{u}_{\varepsilon} \cdot \D_{\alpha} u_{\varepsilon} + \D_{\beta} u_{\varepsilon} \cdot \D_{\alpha} \bar{u}_{\varepsilon}) \, d \mu\\
& \quad + \int_{\h^2} - \Delta (\psi_{\varepsilon} \Delta a + \D_{\alpha} \psi_{\varepsilon} \D^{\alpha} a) \cdot (u_{\varepsilon} \bar{u}_{\varepsilon}) \, d\mu = A + B
\end{align*}
since $\D_{\beta} \bar{u}_{\varepsilon} \cdot \D_{\alpha} \D^{\beta} u_{\varepsilon} +\D_{\beta} u_{\varepsilon} \cdot \D_{\alpha} \D^{\beta} \bar{u}_{\varepsilon} = \D_{\alpha} (\D_{\beta} u_{\varepsilon} \D^{\beta} \bar{u}_{\varepsilon})$. We write $f_{\varepsilon} = \abs{u_{\varepsilon}}^2 u_{\varepsilon} + G_{\varepsilon} = \abs{u_{\varepsilon}}^2 u_{\varepsilon} + g_{\varepsilon} + \NN_{\varepsilon}$ and use the identity $\abs{{u}_{\varepsilon}}^2 \bar{u}_{\varepsilon} \cdot \D_{\alpha} u_{\varepsilon} + \abs{u_{\varepsilon}}^2 u_{\varepsilon} \cdot \D_{\alpha} \bar{u}_{\varepsilon} = \frac{1}{2} \D_{\alpha} (\abs{u_{\varepsilon}}^4)$ to compute
\begin{align*}
II & = 2 \int_{\h^2} \psi_{\varepsilon} \D^{\alpha} a (\bar{f}_{\varepsilon} \cdot \D_{\alpha} u_{\varepsilon} + f_{\varepsilon} \cdot \D_{\alpha} \bar{u}_{\varepsilon}) + \D_{\alpha} (\psi_{\varepsilon} \D^{\alpha} a) \cdot (\bar{f}_{\varepsilon} u_{\varepsilon} + f_{\varepsilon} \bar{u}_{\varepsilon}) \, d \mu\\
& = 2 \int_{\h^2} \psi_{\varepsilon} \D^{\alpha} a (\bar{g}_{\varepsilon} \cdot \D_{\alpha} u_{\varepsilon} + g_{\varepsilon} \cdot \D_{\alpha} \bar{u}_{\varepsilon}) + \D_{\alpha} (\psi_{\varepsilon} \D^{\alpha} a) \cdot (\bar{g}_{\varepsilon} u_{\varepsilon} + g_{\varepsilon} \bar{u}_{\varepsilon}) \, d \mu\\
& + \quad 2 \int_{\h^2} \psi_{\varepsilon} \D^{\alpha} a (\bar{\NN}_{\varepsilon} \cdot \D_{\alpha} u_{\varepsilon} + \NN_{\varepsilon} \cdot \D_{\alpha} \bar{u}_{\varepsilon}) + \D_{\alpha} (\psi_{\varepsilon} \D^{\alpha} a) \cdot (\bar{\NN}_{\varepsilon} u_{\varepsilon} + \NN_{\varepsilon} \bar{u}_{\varepsilon}) \, d \mu \\
& \quad +  \int_{\h^2} \D_{\alpha} (\psi_{\varepsilon} \D^{\alpha} a) \cdot \abs{u_{\varepsilon}}^4 \, d \mu = C +  E + D 
\end{align*}
We integrate these identities on the interval $[t_1, t_2]$ to conclude that
\begin{align*}
M_a (t_2) - M_a (t_1) = \int_{t_1}^{t_2} (A + B + C + E + D) \, dt.
\end{align*}
We use now that $u \in S_{p_{\sigma}}^1 (-T,T)$, thus $\lim_{\varepsilon \to 0} \norm{u_{\varepsilon} - u}_{S_{p_{\sigma}}^1 (-T', T')} =0$ and, using 
\begin{align*}
\norm{(-\Delta)^{\frac{1}{2}} f}_{L^{p_1} (I \times \h^d)} + \norm{f}_{L^{p_2} (I \times \h^d)} \lesssim \norm{f}_{S_q^1 (I)},
\end{align*}
for any $f \in S_q^1 (I)$, $p_1 \in [q, \frac{2d+4}{d}$, and $p_2 \in [q, \frac{2d+4}{d-2})$, we have that
\begin{align*}
\lim_{\varepsilon \to 0} \norm{g_{\varepsilon}}_{L^{p_{\sigma}'} ((-T' , T') \times \h^d)} =0
\end{align*}
for any $T' < T$. We let $\varepsilon \to 0$, using Lemma \ref{lem H a} to conclude that
\begin{align*}
\lim_{\varepsilon \to 0} \int_{t_1}^{t_2} A \, dt & = 2 \int_{[t_1 , t_2] \times \h^2} \D^{\beta} \D^{\alpha} a \cdot (\D_{\beta} \bar{u} \cdot \D_{\alpha} u + \D_{\beta} u \cdot \D_{\alpha} \bar{u}) \, d \mu dt \\
\lim_{\varepsilon \to 0} \int_{t_1}^{t_2} B \, dt & = \lim_{\varepsilon \to 0} \int_{t_1}^{t_2} C \, dt =0,\\
\lim_{\varepsilon \to 0} \int_{t_1}^{t_2} D \, dt & = \int_{[t_1 , t_2] \times \h^2} \abs{u}^4 \, d \mu dt\\
\lim_{\varepsilon \to 0} \int_{t_1}^{t_2} E \, dt & = 2 \int_{[t_1 , t_2] \times \h^2}  \D^{\alpha} a (\bar{\NN} \cdot \D_{\alpha} u + \NN \cdot \D_{\alpha} \bar{u}) + \D_{\alpha} \D^{\alpha} a \cdot (\bar{\NN} u + \NN \bar{u}) \, d \mu 
\end{align*}
Since $\abs{M_a (t)} \leq C \sup_{t \in [t_1 , t_2]} \norm{u(t)}_{L_x^2 (\h^2)} \norm{u(t)}_{H_x^1 (\h^2)}$ using Lemma \ref{lem H a}, it follows that 
\begin{align*}
& \quad 2 \int_{[t_1 , t_2] \times \h^2} \D^{\beta} \D^{\alpha} a \cdot (\D_{\beta} \bar{u} \cdot \D_{\alpha} u + \D_{\beta} u \cdot \D_{\alpha} \bar{u}) \, d \mu dt +  \int_{[t_1 , t_2] \times \h^2} \abs{u}^4 \, d \mu dt\\
& \leq C \sup_{t \in [t_1 , t_2]} \norm{u(t)}_{L_x^2 (\h^2)} \norm{u(t)}_{H_x^1 (\h^2)} +   \norm{\NN \nabla \bar{u}}_{L_{t,x}^1 ([t_1 , t_2] \times \h^2)} +\norm{\NN  \bar{u}}_{L_{t,x}^1 ([t_1 , t_2] \times \h^2)} 
\end{align*}
\end{proof}

\section{Global well-posedness and scattering on $\h^2$}\label{sec GWP+S}
In this section, we use a bootstrapping  argument to finally show the global well-posedness and scattering results stated in Theorem \ref{thm Main H^2}.
\subsection{Step 1: set-up of the open-close argument}
Define
\begin{align*}
W : = \bracket{ T :  \norm{u}_{L_{t,x}^4 ([0, T] \times \h^2)}^4 \leq M},
\end{align*}
where $M>0$ is a constant. $W$ is closed and non-empty. Now we want to show that $W$  is open. If $T_1 \in W$, then due to the local well-posedness theory and Remark \ref{rmk u H^s}, for some $T_0 > T_1$ and $T_0$ sufficiently close to $T_1$ we have
\begin{align*}
 \norm{u}_{L_{t,x}^4 ([0, T_0] \times \h^2)}^4 \leq 2M .
\end{align*}
In fact, the $H^s$ norm of $u(T_1)$ is bounded, then using a standard local well-posedness argument, we can continue the solution $u$ from time $T_1$ at least for a short time.  Within such short time period, due to the sub-criticality, the spacetime $L^4$ norm of $u$ will be bounded by twice the $H^{\sigma}$, for $\sigma$ arbitrarily small,  norm at $T_1$, which is of order $s_0^{\frac{1}{2} (s-1) \sigma}$ (see Remark \ref{rmk u H^s}). Hence we want to ensure that $s_0^{\frac{1}{2} (s-1) \sigma} < M^{\frac{1}{4}} \sim s_0^{\frac{1}{4}(-\frac{1}{2}s + \frac{1}{4})}$, which is achieved for any $s>\frac{3}{4}$ by taking $\sigma$ small enough (say, $\sigma = \frac{1}{4}$). This guarantees the existence of such $T_0$.

Now we show that $T_0 \in W$, that is
\begin{align}\label{eq bootstrap}
 \norm{u}_{L_{t,x}^4 ([0, T_0] \times \h^2)}^4 \leq  M .
\end{align}

Recall the decomposition of  the solution $u$. That is, we can think of $u=\psi + \zeta$, where $\psi = e^{it\Delta} P_{\leq s_0} \phi $ solves a linear Schr\"odinger equation with high frequency data
\begin{align*}
\begin{cases}
i \partial_t \psi + \Delta_{\h^2} \psi= 0, \\
\psi(0,x) =P_{\leq s_0} \phi ,
\end{cases}
\end{align*}
and $\zeta$ solves a difference equation with low frequency data
\begin{align*}
\begin{cases}
i \partial_t \zeta + \Delta_{\h^2} \zeta = \abs{u}^2u =G(\zeta, \psi) , \\
\zeta(0,x) = P_{> s_0} \phi ,
\end{cases}
\end{align*}
where $G(\zeta, \psi) = \abs{\zeta+\psi}^2(\zeta+\psi) = \abs{\zeta}^2 \zeta + \OO(\zeta^2 \psi) +\OO(\zeta \psi^2) + \OO(\psi^3)$. 
From Proposition \ref{prop Global},  we learned that we can divide $[0,T_0] = \cup_i I_i = \cup_i [a_i , a_{i+1}] $, such that on each $I_i$, 
\begin{align}\label{eq u L^4}
\norm{u}_{L_{t,x}^4 (I_i \times \h^2)}^4 = \varepsilon ,
\end{align}
and 
\begin{align*}
\# I_i \sim \frac{2M}{\varepsilon}.
\end{align*}
The total energy increment on $[0,T_0]$ is
\begin{align*}
E(\zeta (T_0) ) \leq E(\zeta(0)) + C \frac{2M}{\varepsilon}  s_0^{\frac{3}{2}s -\frac{5}{4}}  ,
\end{align*}
and the choice of $s_0$ is based on $M$
\begin{align*}
M \sim s_0^{-\frac{1}{2}s + \frac{1}{4}} .
\end{align*}
Using the smallness of $L^4 $ norm of $\psi$ in \eqref{eq psi L^4}
\begin{align*}
\norm{\psi}_{L_{t,x}^4([0,T_0] \times \h^2)} \lesssim   s_0^{\frac{1}{2}s} ,
\end{align*}
one can reduce \eqref{eq bootstrap} to 
\begin{align}\label{eq bootstrap zeta}
 \norm{\zeta}_{L_{t,x}^4 ([0, T_0] \times \h^2)}^4 \leq  \frac{1}{2} M .
\end{align}

Now we will prove the improved  bound of  $L_{t,x}^4$ in \eqref{eq bootstrap zeta} in Steps 2 and 3.

\subsection{Step 2: improving the bound for $L_{t,x}^4$}
Recall the modified Morawetz estimate in \eqref{modmor} that now gives
\begin{align}\label{eq Morawetz}
\begin{aligned}
\norm{\zeta}_{L_{t,x}^4([0,T_0] \times \h^2)}^4 & \lesssim \norm{\zeta}_{L_t^{\infty} L_x^2 ([0,T_0] \times \h^2)} \norm{\zeta}_{L_t^{\infty} H_x^1 ([0,T_0] \times \h^2)} \\
& \quad + \norm{\NN\bar{\zeta}}_{L_{t,x}^1 ([0,T_0] \times \h^2)} +  \norm{\NN \nabla \bar{ \zeta}}_{L_t^1L_{x}^{1} ([0,T_0] \times \h^2)} .
\end{aligned}
\end{align}

In our case  $\NN$ is given by
\begin{align*}
\NN & = \abs{u}^2 u - \abs{\zeta}^2 \zeta \\
& = \abs{\psi + \zeta}^2 (\psi + \zeta) - \abs{\zeta}^2 \zeta =  \abs{\psi}^2 \psi + (2 \abs{\psi}^2 \zeta + \psi^2 \bar{\zeta}) + (2 \psi \abs{\zeta}^2 + \bar{\psi} \zeta^2) \\
& = \abs{\psi}^2 \psi + \OO(\psi^2 \zeta) + \OO(\psi \zeta^2) .
\end{align*}

Now we estimate the right hand side terms in \eqref{eq Morawetz}. 

For the second term in \eqref{eq Morawetz}, by H\"older inequality, \eqref{eq psi L^4}  and \eqref{eq bootstrap zeta}, we have
\begin{align*}
\norm{\NN \bar{\zeta}}_{L_{t,x}^1 ([0, T_0] \times \h^2)} & \lesssim  \norm{\psi}_{L_{t,x}^4([0,T_0] \times \h^2)}^3 \norm{\zeta}_{L_{t,x}^4([0,T_0] \times \h^2)}  +\norm{\psi}_{L_{t,x}^4([0,T_0] \times \h^2)}  \norm{\zeta}_{L_{t,x}^4([0,T_0] \times \h^2)}^3  \\
& \lesssim s_0^{\frac{3}{2}s} \norm{\zeta}_{L_{t,x}^4([0,T_0] \times \h^2)}  + s_0^{\frac{1}{2}s} \norm{\zeta}_{L_{t,x}^4([0,T_0] \times \h^2)}^3  \lesssim s_0^{\frac{3}{2}s} M^{\frac{1}{4}} + s_0^{\frac{1}{2}s} M^{\frac{3}{4}} .
\end{align*}

We write the last term in \eqref{eq Morawetz} as
\begin{align}
\norm{\NN \nabla \bar{ \zeta}}_{L_{t,x}^1 ([0,T_0] \times \h^2)} & \lesssim \norm{\OO(\zeta^2 \psi ) \nabla \bar{ \zeta}}_{L_{t,x}^1 ([0,T_0] \times \h^2)} + \norm{\OO(\psi^3) \nabla \bar{ \zeta}}_{L_{t,x}^1 ([0,T_0] \times \h^2)} \notag\\
& \lesssim \norm{\zeta}_{L_{t,x}^4 ([0,T_0] \times \h^2)}^2 \norm{\psi}_{L_{t,x}^4 ([0,T_0] \times \h^2)}  \norm{\nabla \zeta}_{L_{t,x}^4 ([0,T_0] \times \h^2)} \label{eq error}\\
& \quad + \norm{\psi}_{L_{t,x}^4 ([0,T_0] \times \h^2)}^3 \norm{\nabla  \zeta}_{L_{t,x}^4 ([0,T_0] \times \h^2)} \notag .
\end{align}

\begin{claim}\label{claim Scattering}
We claim that
\begin{enumerate}[\it (1)]
\item
$\norm{\nabla \zeta}_{L_{t,x}^4 ([0,T_0] \times \h^2)}^4   \lesssim  M s_0^{2(s-1)}$,

\item
$\norm{\nabla \zeta}_{L_{t}^{\infty} L_x^2 ([0, T_0] \times \h^2)}  \lesssim s_0^{\frac{1}{2}(s-1)} $.
\end{enumerate}
\end{claim}
Assuming Claim \ref{claim Scattering}, we continue the estimation  of \eqref{eq error},
\begin{align*}
\eqref{eq error} \lesssim M^{\frac{1}{2}} s_0^{\frac{1}{2}s} (M s_0^{2(s-1)})^{\frac{1}{4}} + s_0^{\frac{3}{2} s} (M s_0^{2(s-1)})^{\frac{1}{4}} = M^{\frac{3}{4}} s_0^{s -\frac{1}{2}} + M^{\frac{1}{4}} s_0^{2s -\frac{1}{2}} .
\end{align*}

Using the same calculation as in \eqref{eq zeta L^2}, we have 
\begin{align*}
\norm{\zeta (t)}_{L_x^2 ( \h^2)} & \leq \norm{u (t)}_{L_x^2 ( \h^2)} + \norm{\psi (t)}_{L_x^2 ( \h^2)}  \lesssim \norm{u (0)}_{L_x^2 ( \h^2)} + s_0^{\frac{1}{2}s} \leq 2\norm{u (0)}_{L_x^2 ( \h^2)} ,
\end{align*}
hence the first term in \eqref{eq Morawetz} is bounded by
\begin{align*}
\norm{\zeta}_{L_t^{\infty} L_x^2 ([0,T_0] \times \h^2)} \norm{\zeta}_{L_t^{\infty} H_x^1 ([0,T_0] \times \h^2)} \lesssim \norm{u (0)}_{L_x^2 ( \h^2)} s_0^{\frac{1}{2} (s-1)} .
\end{align*}

Now \eqref{eq Morawetz} becomes
\begin{align*}
\norm{\zeta}_{L_{t,x}^4([0,T_0] \times \h^2)}^4  \lesssim s_0^{\frac{1}{2}(s-1)} + (s_0^{\frac{3}{2}s} M^{\frac{1}{4}}  + s_0^{\frac{1}{2}s} M^{\frac{3}{4}}) +  (M^{\frac{3}{4}} s_0^{s -\frac{1}{2}} + M^{\frac{1}{4}} s_0^{2s -\frac{1}{2}}) .
\end{align*}

To close the argument, we need the following inequality holds for $M \sim s_0^{-\frac{1}{2} s + \frac{1}{4}}$
\begin{align}\label{eq requirement} 
s_0^{\frac{1}{2}(s-1)} + (s_0^{\frac{3}{2}s} M^{\frac{1}{4}}  + s_0^{\frac{1}{2}s} M^{\frac{3}{4}}) +  (M^{\frac{3}{4}} s_0^{s -\frac{1}{2}} + M^{\frac{1}{4}} s_0^{2s -\frac{1}{2}}) < \frac{1}{2} M .
\end{align}
This requirement of \eqref{eq requirement} can be achieved for 
\begin{align}\label{eq H s}
s > \frac{3}{4}.
\end{align}

Now we are left to prove Claim \ref{claim Scattering}.

\subsection{Step 3: Proof of Claim \ref{claim Scattering}}
\begin{proof}[Proof of Claim \ref{claim Scattering}]
We start with {\it (2)}. Recall the the total energy increment,
\begin{align*}
E(\zeta(t)) & \leq E(\zeta(0)) + \frac{M}{\varepsilon} s_0^{\frac{3}{2}s -\frac{5}{4}} \sim s_0^{s-1}  ,
\end{align*}
for all $t \in [0, T_0]$.
This yields {\it (2)}
\begin{align*}
\norm{\nabla \zeta}_{L_{t}^{\infty} L_x^2 ([0, T_0] \times \h^2)}^2 & \leq \sup_t E(\zeta(t)) \lesssim s_0^{s-1} .
\end{align*}

To estimate $\norm{\nabla \zeta}_{L_t^4 L_x^{4} ([0,T_0] \times \h^2)}$, we consider the subintervals $I_i$'s. We claim that 
\begin{align}\label{eq zeta nabla L^4 bdd}
 \norm{\nabla \zeta}_{L_{t,x}^4 (I_i \times \h^2)}   \lesssim \norm{\nabla \zeta(a_i)}_{L_x^2 (\h^2)}.
\end{align}
In fact, by Strichartz estimates
\begin{align}
\norm{\nabla \zeta}_{L_{t,x}^4 (I_i \times \h^2)} & \lesssim \norm{\nabla \zeta(a_i)}_{L_x^2 (\h^2)} + \norm{\nabla (\zeta + \psi)^3}_{L_{t,x}^{\frac{4}{3}} (I_i \times \h^2)} \notag\\
& \lesssim \norm{\nabla \zeta(a_i)}_{L_x^2 (\h^2)} + \norm{\nabla \OO(\zeta^3)}_{L_{t,x}^{\frac{4}{3}} (I_i  \times \h^2)} + \norm{\nabla \OO(\psi^3)}_{L_{t,x}^{\frac{4}{3}} (I_i  \times \h^2)} . \label{eq zeta nabla L^4} 
\end{align}
The second term in \eqref{eq zeta nabla L^4} will be absorbed by the left hand side of  \eqref{eq zeta nabla L^4}
\begin{align*}
\norm{\nabla \OO(\zeta^3)}_{L_{t,x}^{\frac{4}{3}} (I_i  \times \h^2)} \lesssim  \norm{\nabla \zeta}_{L_{t,x}^{4} (I_i \times \h^2)} \norm{\zeta}_{L_{t,x}^{4} (I_i \times \h^2)}^2 \lesssim \norm{\nabla \zeta}_{L_{t,x}^{4} (I_i  \times \h^2)} \varepsilon^{\frac{1}{2}} .
\end{align*}
The last inequality above is due to \eqref{eq psi L^4} and \eqref{eq u L^4}
\begin{align}\label{eq zeta L^4}
\norm{\zeta}_{L_{t,x}^4 (I_i \times \h^2)} \leq  \norm{\psi}_{L_{t,x}^4 (I_i \times \h^2)} + \norm{u}_{L_{t,x}^4 (I_i \times \h^2)} \lesssim  s_0^{\frac{1}{2} s} + \varepsilon^{\frac{1}{4}}  \lesssim \varepsilon^{\frac{1}{4}}.
\end{align}
For the last term in \eqref{eq zeta nabla L^4}, the same calculation as in \eqref{eq psi^3} gives 
\begin{align*}
\norm{\nabla \OO(\psi^3)}_{L_{t,x}^{\frac{4}{3}} (I_i  \times \h^2)}  \lesssim s_0^{\frac{3}{2}s -\frac{3}{4}} .
\end{align*}

Then \eqref{eq zeta nabla L^4} becomes
\begin{align*}
\norm{\nabla \zeta}_{L_{t,x}^4 (I_i \times \h^2)}  \lesssim  \norm{\nabla \zeta(a_i)}_{L_x^2 (\h^2)} + \norm{\nabla \zeta}_{L_{t,x}^{4} (I_i  \times \h^2)} \varepsilon^{\frac{1}{2}} + s_0^{\frac{3}{2}s -\frac{3}{4}} .
\end{align*}
Therefore the claim \eqref{eq zeta nabla L^4 bdd} follows.

Putting all the small intervals together and using {\it (2)} we get
\begin{align*}
\norm{\nabla \zeta}_{L_{t,x}^4 ([0,T_0] \times \h^2)}^4   \lesssim \frac{M}{\varepsilon} \sup_{I_i}\norm{\nabla \zeta(a_i)}_{L_x^2 (\h^2)}^4 \leq \frac{M}{\varepsilon} \norm{\nabla \zeta}_{L_{t}^{\infty} L_x^2 ([0, T_0] \times \h^2)}^4 \lesssim \frac{M}{\varepsilon} s_0^{2(s-1)}.
\end{align*}

Now we finish the proof of Claim \ref{claim Scattering}.

\end{proof}

\subsection{Step 4: proof of scattering}
Recall the definition of scattering: given a global solution $u \in H^s$ to \eqref{NLS}, we say that $u$ scatters to $u_{\pm} \in H^s$ if
\begin{align*}
\lim_{t \to \pm \infty} \norm{u(t) - e^{it\Delta_{\h^2}} u_{\pm}}_{H_x^s (\h^2)} =0.
\end{align*}

It is clear that scattering is equivalent to showing that the improper time integral
\begin{align*}
\int_0^{\infty} e^{-it'\Delta} \abs{u}^2 u (t') \, dt'
\end{align*}
converges in $H^s$ and in particular this will give the formula for $u_+$, that is
\begin{align*}
u_+ = u_0 - i\int_0^{\infty} e^{-it'\Delta} \abs{u}^2 u (t') \, dt' .
\end{align*}

By Strichartz and Lemma \ref{lem G Product rule}, we have that
\begin{align*}
\norm{\int_0^{\infty} e^{-it'\Delta} \abs{u}^2 u (t') \, dt'}_{H_x^s (\h^2)} \lesssim \norm{ \inner{-\Delta}^{\frac{s}{2}} (\abs{u}^2 u)}_{L_{t,x}^{\frac{4}{3}} (\R \times \h^2)} \lesssim \norm{\inner{-\Delta}^{\frac{s}{2}}u}_{L_{t,x}^4 (\R \times \h^2)} \norm{u}_{L_{t,x}^4 (\R \times \h^2)}^2 .
\end{align*}
It is clear that the scattering follows once we show that 
\begin{align*}
\norm{u}_{S^s (\R)}  \leq  C.
\end{align*}
Moreover, we can reduce to prove
\begin{align*}
\norm{\zeta}_{S^s (\R)}  \leq  C,
\end{align*}
since we learned in \eqref{eq psi S^sigma} that $\norm{\psi}_{S^s (\R)} \lesssim 1$.

We proved in Step 3 that $\norm{u}_{L_{t,x}^4(\R \times \h^2)} \leq C$, so we divide the time interval $(-\infty,\infty)$ into $\cup I_i = \cup [a_i, a_{i+1}]$, $i=1, \cdots , K < \infty$ such that
\begin{align*}
\norm{u}_{L_{t,x}^4(I_i \times \h^2)}^4 = \varepsilon 
\end{align*}
for all $i = 1, \cdots, K$.

On each $I_i = [a_i, a_{i+1}]$, by the same calculation as in \eqref{eq zeta nabla L^4}
\begin{align*}
\norm{\nabla \zeta}_{S^0 (I_i )} & \lesssim  \norm{\zeta(a_i)}_{H_x^1(\h^2)} + \norm{\nabla \OO(\zeta^3)}_{L_{t,x}^{\frac{4}{3}} (I_i \times \h^2)} + \norm{\nabla \OO(\psi^3)}_{L_{t,x}^{\frac{4}{3}} (I_i \times \h^2)}\\
& \lesssim  \norm{\zeta (a_i)}_{H_x^1(\h^2)} +  \varepsilon^{\frac{1}{2}} \norm{\nabla \zeta}_{S^0 (I_i )}   + s_0^{\frac{3}{2}s -\frac{3}{4}} .
\end{align*}
Then we have
\begin{align*}
\norm{\nabla \zeta}_{S^0 (I_i )}  \lesssim   \norm{\zeta(a_i)}_{H_x^1(\h^2)} .
\end{align*}
Therefore, due to the finiteness of number of $I_i$ intervals,
\begin{align*}
\norm{\nabla \zeta}_{S^0 (\R)}  \leq  C.
\end{align*}
Using the integral equation and Strichartz estimates again with \eqref{eq zeta L^4} and \eqref{eq psi L^4}, we have
\begin{align*}
\norm{ \zeta}_{S^0 (I_i )} & \lesssim  \norm{\zeta(a_i)}_{L_x^2(\h^2)} + \norm{ \OO(\zeta^3)}_{L_{t,x}^{\frac{4}{3}} (I_i \times \h^2)} + \norm{\OO(\psi^3)}_{L_{t,x}^{\frac{4}{3}} (I_i \times \h^2)}\\
& \lesssim  \norm{\zeta (a_i)}_{L_x^2(\h^2)} +   \norm{\zeta}_{L_{t,x}^{4} (I_i \times \h^2)}^3 +   \norm{\psi}_{L_{t,x}^{4} (I_i \times \h^2)}^3 \lesssim  \norm{\zeta (a_i)}_{L_x^2(\h^2)} +  \varepsilon^{\frac{3}{4}} +   s_0^{\frac{3}{2}s} .
\end{align*}
Then
\begin{align*}
\norm{\zeta}_{S^0 (I_i )}  \lesssim   \norm{\zeta(a_i)}_{L_x^2(\h^2)} ,
\end{align*}
and due to the finiteness of number of $I_i$ intervals,
\begin{align*}
\norm{\zeta}_{S^0 (\R)}  \leq  C .
\end{align*}
Therefore, interpolating $S^0$ with $S^1$ gives
\begin{align*}
\norm{\zeta}_{S^s (\R)}  \leq  C .
\end{align*}

We finish the proof of scattering.

\section{General nonlinearities}\label{general}
Our result for the cubic NLS in Theorem \ref{thm Main H^2}  can be generalized to a larger class of  nonlinearities. In fact we have the following result. 
\begin{thm}\label{thm Main H^2 p}
The initial value problem \eqref{NLSHp} with radial initial data $\phi \in H^s(\h^2)$  is globally-well-posed and scatters in $H^s(\h^2)$ when $s > \frac{3p-6}{3p-5} $.
\end{thm}
Note that the scaling of \eqref{NLSHp} is $s_c =1- \frac{2}{p-1}$.

\subsection{Sketch of the proof}
We will briefly present how the method presented above for the cubic NLS can be generalized to nonlinearities of order $p$. 
\begin{enumerate}
\item
As in Section \ref{sec Energy increment}, we decompose $u$ in $\psi$ and $\zeta$, where $\psi = e^{it\Delta} P_{\leq s_0} \phi $ solves the linear Schr\"odinger with high frequency data and $\zeta$ solves the difference equation with low frequency data
\begin{align}\label{eq psi zeta p}
\begin{cases}
i \partial_t \psi + \Delta_{\h^2} \psi= 0,  \\
\psi(0,x) =\psi_0 = P_{\leq s_0} \phi ,  
\end{cases}
&&&
\begin{cases}
i \partial_t \zeta + \Delta_{\h^2} \zeta = \abs{u}^{p-1 }u , \\
\zeta(0,x) = \eta_0 = P_{> s_0} \phi .
\end{cases}
\end{align}
Then using similar analysis we obtain the global energy increment given the boundedness of $u$ in  the critical spacetime $L_{t,x}^{2(p-1)},$ (see Proposition \ref{prop Local p} and  Proposition \ref{prop Global p} below).  The  analogues of all the estimates that we used in the cubic case can be found in \eqref{eq All estimates p}. 
\item
Similarly, a bootstrapping argument on the $L_{t,x}^{2(p-1)}$ norm gives both the global existence and scattering.
\end{enumerate}
Note that the only difference in the general case is that  the spacetime $L_{t,x}^{2(p-1)}$ in the local theory (Proposition \ref{prop Local p}) is different from the Morawetz norm. Hence in the  bootstrapping argument an intermediate step is needed. In fact, in this step,  we first obtain and improve the estimates on the Morawetz norm,  then we bootstrap the $L_{t,x}^{2(p-1)}$ norm with the better  Morawetz bound. Notice that $L_{t,x}^{2(p-1)}$ agrees with the Morawetz norm  when $p=3$, hence such step is not needed in Section \ref{sec GWP+S}.

\subsection{Analogues of the main propositions}
We now present the analogues of Proposition \ref{prop Local} and Proposition \ref{prop Global} on the energy increment.
\begin{prop}[Local energy increment]\label{prop Local p}
Consider $u$ as in \eqref{NLSHp} defined on $I \times \h^2$ where $I = [0,\tau]$, such that
\begin{align*}
\norm{u}_{L_{t,x}^{2(p-1)} (I \times \h^2)}^{2(p-1)} = \varepsilon
\end{align*}
for some universal constant $\varepsilon$. Then for $s > \frac{p}{p+1}$ and sufficiently small $s_0$, the solution $\zeta$, under the decomposition $u = \psi + \zeta$ defined as in \eqref{eq psi zeta p}, satisfies the following energy increment
\begin{align*}
E(\zeta (\tau) ) \leq E(\zeta(0)) + C  s_0^{\frac{p+3}{4}s -\frac{p+2}{4}}.
\end{align*}
\end{prop}

\begin{prop}[Conditional global energy increment]\label{prop Global p}
Consider $u$ as in \eqref{NLSHp} defined on $[0, T] \times \h^2$ where
\begin{align*}
\norm{u}_{L_{t,x}^{2(p-1)}  ([0,T] \times \h^2)}^{2(p-1)}  \leq M
\end{align*}
for some constant $M$. Then for $s > \frac{p}{p+1}$ and sufficiently small $s_0$, the energy of $\zeta$ satisfies the following energy increment
\begin{align*}
E(\zeta (T) ) \leq E(\zeta(0)) + C \frac{M}{\varepsilon}  s_0^{\frac{p+3}{4}s -\frac{p+2}{4}}.
\end{align*}
where $\varepsilon$ is the small constant in Proposition \ref{prop Local p}.
\end{prop}

\subsection{Analogues of the main estimates}
Within the proofs of two propositions above, we need a further decomposition for $\zeta$ as is \eqref{eq zeta1}, \eqref{eq zeta2} and  \eqref{eq u sum}
\begin{align}\label{eq zeta1 zeta2 p}
\begin{cases}
i \partial_t \zeta_1 + \Delta_{\h^2} \zeta_1 = \abs{\zeta_1}^{p-1} \zeta_1 , \\
\zeta_1(0,x) = \eta_0 = P_{> s_0} \phi ,
\end{cases}
&&&
\begin{cases}
i \partial_t \zeta_2 + \Delta_{\h^2} \zeta_2 = \abs{u}^{p-1} u - \abs{\zeta_1}^{p-1} \zeta_1 , \\
\zeta_2(0,x) =0.
\end{cases}
\end{align}
Hence  the full solution $u$ is the sum of these three solutions $u = \zeta_1 + \zeta_2 + \psi$.

The analogue of \eqref{eq All estimates} can be computed similarly as follows
\begin{flalign}\label{eq All estimates p}
&\norm{\psi}_{S^{\sigma}(\R)} \lesssim s_0^{\frac{1}{2}(s-\sigma)} \text{ for } 0 \leq \sigma \leq s ,&&  \norm{\psi}_{L_{t,x}^{2(p-1)} (\R \times \h^2)}  \lesssim s_0^{\frac{1}{2}(s-s_c)} , & \notag\\
&\norm{\zeta_1}_{S^{\sigma} (I )} \lesssim_{\norm{\phi}_{H_x^{s_c}}}
\begin{cases}
1 & \text{ for } 0 \leq \sigma \leq s_c ,\\
s_0^{\frac{\sigma-s_c}{2(1-s_c)}(s-1)} & \text{ for } s_c \leq \sigma \leq 1 ,
\end{cases}
&& \norm{\zeta_1}_{L_{t,x}^{2(p-1)} (I \times \h^2)}^{2(p-1)}  \lesssim_{\norm{\phi}_{H_x^{s_c}}} \varepsilon  , &\\
&\norm{\zeta_2}_{S^{\sigma}(I)} \lesssim_{\norm{\phi}_{H_x^{s_c}}} s_0^{\frac{1}{2} (s-\sigma)} \text{ for } 0 \leq \sigma \leq s ,&& \norm{\zeta_2}_{L_{t,x}^{2(p-1)} (I \times \h^2)} \lesssim_{\norm{\phi}_{H_x^{s_c}}} s_0^{\frac{1}{2}(s-s_c)} . & \notag
\end{flalign}
Claim \ref{claim useful norms} will be the same in the general setting. Most importantly, the analogue of {\it (3)} in Lemma \ref{lem Est zeta2} is
\begin{align*}
\norm{\zeta_2}_{L_t^{\infty} H_x^1 (I \times \h^2)} \lesssim  s_0^{\frac{p+1}{4}s -\frac{p}{4}} .
\end{align*}
The choice of $s_0$ in Proposition \ref{prop Global p} is given by
\begin{align}\label{eq Mp}
M \sim s_0^{\frac{1}{2}(\frac{1-s}{1-s_c} -\frac{1}{2})} .
\end{align}
It is also worth mentioning that the hidden constant in \eqref{eq All estimates p} at the second iteration is bounded by the $H^{s_c}$ norm of $\phi$ plus a small error
\begin{align*}
\norm{(\zeta_1 + \zeta_2) (a_1)}_{H_x^{s_c} (\h^2)} & \leq \norm{\zeta_1 (a_1)}_{H_x^{s_c} (\h^2)} + \norm{ \zeta_2 (a_1)}_{H_x^{s_c} (\h^2)}  \lesssim \norm{\phi}_{H_x^{s_c} (\h^2)} + s_0^{\frac{1}{2}s} s_0^{-\frac{1}{2} s_c (\frac{1}{2} + \frac{1-s}{1-s_c})} ,
\end{align*}
for $s > \frac{p}{p+1}$.
Then the accumulated gain of this hidden constant will be dominated by the size of the $H^{s_c}$ norm of $\phi$, hence not growing.

\subsection{A different bootstrapping argument}
We consider
\begin{align*}
W : = \bracket{ T :  \norm{u}_{L_{t,x}^{2(p-1)} ([0, T] \times \h^2)}^{2(p-1)} \leq M},
\end{align*}
We are then reduced to showing  that for $T_0$ chosen in the same manner as in Step 1 in Section \ref{sec GWP+S}
\begin{align*}
\norm{\zeta}_{L_{t,x}^{2(p-1)} ([0, T_0] \times \h^2)}^{2(p-1)} \leq  \frac{1}{2} M .
\end{align*}
Things are different here. First, interpolating $\norm{\zeta}_{L_{t,x}^{2+} ([0, T_0] \times \h^2)}^{2+} \lesssim M$ with the bound of the $L_{t,x}^{2(p-1)}$ norm in the assumption gives an estimate on the Morawetz norm
\begin{align}\label{eq Mora}
\norm{\zeta}_{L_{t,x}^{p+1} ([0, T_0] \times \h^2)}^{p+1} \lesssim M.
\end{align}
Using \eqref{eq Mora} and the modified Morawetz estimate \eqref{modmor}, we obtain as before
\begin{align*}
 \norm{\zeta}_{L_{t,x}^{p+1} ([0, T_0] \times \h^2)}^{p+1} \lesssim s_0^{\frac{1}{2}(s-1)}  .
\end{align*}
If we simply require $s_0^{\frac{1}{2}(s-1)} < M$ here, there will be no room to improve the $L_{t,x}^{2(p-1)}$ norm at all. So to this end, we demand it to be much smaller than $M$,  that is for $\alpha \in (0,1)$ 
\begin{align}\label{eq Bootimprove}
s_0^{\frac{1}{2}(s-1)}  \leq M^{\alpha} \ll M,
\end{align}
hence recalling \eqref{eq Mp}, we get the first restriction on $s$
\begin{align*}
s > 1- \frac{\alpha}{2+ \alpha (p-1)} .
\end{align*}
With this better Morawetz bound, we can improve the $L_{t,x}^{2(p-1)}$ norm by making it smaller than $M$ by H\"older inequality, \eqref{eq Bootimprove} and Proposition \ref{prop Global p} with the choice of $M$ as in \eqref{eq Mp},
\begin{align*}
\norm{\zeta}_{L_{t,x}^{2(p-1)} ([0, T_0] \times \h^2) }^{2(p-1)} & \lesssim \norm{\zeta}_{L_{t,x}^{p+1}([0, T_0] \times \h^2) }^{p+1-} \norm{\zeta}_{L_{t,x}^{\infty-} ([0, T_0] \times \h^2) }^{p-3+}  \lesssim M^{\alpha-} \norm{\inner{\nabla}^{1-} \zeta}_{L_{t}^{\infty-} L_x^{2+} ([0, T_0] \times \h^2) }^{p-3+} \\
& \lesssim M^{\alpha } s_0^{\frac{1}{2} (s-1) (p-3)}  \ll M.
\end{align*}
This with \eqref{eq Mp} implies the second restriction on $s$
\begin{align*}
s > 1- \frac{1-\alpha}{2(p-3)+ (1-\alpha)(p-1)} .
\end{align*}
Therefore combining both conditions on $s$, we choose $s > s_{\h^2}^p$ to be the best possible scattering index, where
\begin{align*}
s_{\h^2}^p = \min_{\alpha \in (0,1)} \max \bracket{1- \frac{\alpha}{2+ \alpha (p-1)}  , 1- \frac{1-\alpha}{2(p-3)+ (1-\alpha)(p-1)} } = 1- \frac{1}{3p-5} = \frac{3p-6}{3p-5} .
\end{align*}

\appendix
\section{Global well-posedness result in $\R^2$}
\subsection{Tools used in the proof on $\R^2$}

In this subsection we recall known estimates for the Schr\"odinger operator in $\R^2$.
We start by recalling that a couple $(q,r)$ of exponents is admissible if $(\frac{1}{q},\frac{1}{r})$  belongs to the line 
\begin{align}\label{Id}
I_d = \{(\frac{1}{q},\frac{1}{r}) \in [0,\frac{1}{2}] \times ( 0,\frac{1}{2}] \, \big| \, \frac{2}{q} + \frac{d}{r} = \frac{d}{2} \}.
\end{align} 
Then we have
the following  
\begin{thm}[Strichartz Estimates \cite{GV, Y, KT}] 
Assume $u$ is the solution to the inhomogeneous initial value problem
\begin{align}\label{InS}
\begin{cases}
i \partial_t u + \Delta u = F, & t \in \R , \quad x \in \R^d,\\
u(0,x) = f(x), & 
\end{cases}
\end{align}
For any admissible exponents $(q,r)$ and $(\tilde{q}, \tilde{r})$ we have the Strichartz estimates:
\begin{align*}
\norm{u}_{L_t^q L_x^r(\R \times \R^d)} \lesssim \norm{f}_{L_x^2(\R^d)} + \norm{F}_{L_t^{\tilde{q}'}  L_x^{\tilde{r}'} (\R \times \R^d)} .
\end{align*}
\end{thm}

\begin{defn}[Strichartz Spaces]
We define the Banach space
\begin{align*}
S^0 (I) = \bracket{f \in C(I : L^2(\R^2)) : \norm{f}_{S^0 (I)} = \sup_{(q,r) \text{ admissible }} \norm{f}_{L_t^q L_x^r (I  \times \R^2)} < \infty} .
\end{align*}
Also we define the Banach space $S^{\sigma} (I)$, where $\sigma >0$,
\begin{align*}
S^{\sigma} (I) = \bracket{f \in C(I : H^{\sigma} (\R^2)) : \norm{f}_{S^{\sigma}(I)} = \norm{\langle\nabla^{\sigma}\rangle f}_{S^0(I)} < \infty} .
\end{align*}
\end{defn}

\begin{thm}[Local Smoothing Estimates in $\R^2$ \cite{CS, S, V}]\label{lem LSE}
For any $\varepsilon >0$,
\begin{align*}
& \norm{\inner{x}^{-\frac{1}{2} -\varepsilon} \abs{\nabla}^{\frac{1}{2}} e^{it\Delta} f}_{L_{t,x}^2 (\R \times \R^2)} \lesssim \norm{f}_{L_x^2(\R^2)} ,\\
& \norm{\inner{x}^{-\frac{1}{2} -\varepsilon} \nabla \int_0^t e^{i(t-s) \Delta} F(s,x) \, ds}_{L_{t,x}^2 (\R \times \R^2)} \lesssim \norm{\inner{x}^{\frac{1}{2} +\varepsilon} F}_{L_{t,x}^2(\R \times \R^2)} .
\end{align*}
\end{thm}

\begin{prop}[Radial Sobolev Embeddings in $\R^d$ in \cite{TVZ}]\label{prop Radial Sobolev}
Let $d \geq 1$, $1 \leq q \leq \infty$, $0 < s < d$ and $\beta \in \R$ obey the conditions
\begin{align*}
\beta > -\frac{d}{q}, \quad 0 \leq \frac{1}{p} -\frac{1}{q} \leq s
\end{align*}
and the scaling condition
\begin{align*}
\beta + s = \frac{d}{p} - \frac{d}{q}
\end{align*}
with at most one of the equalities
\begin{align*}
p=1, \quad p=\infty, \quad q=1, \quad q=\infty, \quad \frac{1}{p} - \frac{1}{q} =s
\end{align*}
holding. Then for any spherically symmetric function $f \in \dot{W}^{s,p} (\R^d)$, we have 
\begin{align*}
\norm{\abs{x}^{\beta} f}_{L^q (\R^d)} \lesssim \norm{\abs{\nabla}^s f}_{L^p (\R^d)}.
\end{align*}
\end{prop}

\subsection{Theorem}
If we follow the same blue print set up in the hyperbolic setting, we can prove the following global well-posedness result in $\R^2$.
\begin{thm}\label{thm Main R^2}
The initial value problem 
\begin{align}\label{RNLS}
\begin{cases}
i \partial_t u + \Delta u = \abs{u}^2 u, & t \in \R , x \in \R^2,\\
u(0,x) = \phi(x), & 
\end{cases}
\end{align}
is globally well-posed for radial data $\phi \in H^s(\R^2)$ when $s >\frac{4}{5} $.
\end{thm}

\begin{rmk}
We did not prove the scattering part in $\R^2$ setting, since the Morawetz estimate and the Strichartz estimates are less favorable in two dimensional Euclidean space. More precisely, the Morawetz estimate is significantly different from the one in higher dimensions $\R^d$ ($d \geq 3$), also from the one that we used in $\h^d$. So it is not straightforward to employ the Morawetz in $\R^2$ setting. Also the range of Strichartz admission pairs are limited comparing to that in the hyperbolic space. 
\end{rmk}

\end{document}